\documentclass[a4paper, 10pt, notitlepage]{article}

\usepackage{amsthm, amsmath, amssymb, latexsym}
\usepackage{mathrsfs,cite}

\theoremstyle{plain}
\newtheorem{theorem}{Theorem}

\newtheorem{corollary}{Corollary}
\newtheorem{proposition}{Proposition}

\theoremstyle{definition}
\newtheorem{definition}{Definition}
\newtheorem{example}{Example}

\theoremstyle{remark}
\newtheorem{remark}{Remark}

\DeclareMathOperator*{\argmin}{arg\,min}
\DeclareMathOperator{\rank}{rank}

\author{M.V. Dolgopolik}
\title{Existence of augmented Lagrange multipliers: reduction to exact penalty functions and localization
principle}

\begin{document}

\maketitle

\begin{abstract}
In this article, we present new general results on existence of augmented Lagrange multipliers. We define a
penalty function associated with an augmented Lagrangian, and prove that, under a certain growth assumption on the
augmenting function, an augmented Lagrange multiplier exists if and only if this penalty function is exact. We also
develop a new general approach to the study of augmented Lagrange multipliers called the localization principle. The
localization principle allows one to study the local behaviour of the augmented Lagrangian near globally optimal
solutions of the initial optimization problem in order to prove the existence of augmented Lagrange multipliers.
\end{abstract}

\section{Introduction}

The augmented Lagrangian approach to optimization problems with equality constraints was introduced by Hestenes
\cite{Hestenes} and Powell \cite{Powell}. Due to its efficiency and various interesting features, this approach
became a popular tool of nonlinear optimization, and was thoroughly investigated by many researchers
(see~\cite{Rockafellar1,Rockafellar2,ShapiroSun2004,AndreaniBirginEtAl,RuckmannShapiro,Gasimov,BurachikGasimovEtAl,
BirginFloudas2010,CurtisJianRobinson,ChatzipanagiotisEtAl,BurachikIusemMelo2015,BirginMartinez_book,BirginMartinez2015}
and references therein). A modern general formulation of the augmented Lagrangian method was first proposed by
Rockafellar and Wets \cite{RockWets}, and further developed in \cite{HuangYang2003,ZhouYang2004,HuangYang2005}. Let us
also mention several extensions
\cite{GasimovRubinov,BurachikRubinov,ZhouYang2006,ZhangYang2008,ZhouYang2009,BurachikIusemMelo,ZhouYang2012,
WangYangYang2014,WangLiuQu} of this augmented Lagrangian method aiming at including some other augmented Lagrangian and
penalty methods into the unified framework proposed in \cite{RockWets}. Throughout this article, we use a direct
generalization of the augmented Lagrangian introduced in \cite{RockWets}.

One of the main notions in the theory of augmented Lagrangian functions is an augmented Lagrange multiplier
(see~\cite{ShapiroSun2004}). Apart from the fact that the existence of an augmented Lagrange multiplier guarantees the
absence of duality gap, as well as the existence of a globally optimal solution of the augmented dual problem, the
existence of an augmented Lagrange multiplier is important for the global convergence analysis of primal-dual methods
based on the use of the augmented Lagrangian \cite{BurachikIusemMelo2015}. Thus, the problem of existence of augmented
Lagrange multipliers (which is closely related to the problem of existence of global saddle points of the augmented
Lagrangian) is very important for the theory of augmented Lagrangian methods. 

Many known general necessary and/or sufficient conditions for the existence of augmented Lagrange multipliers are
based either on some abstract assumptions on the optimal value function \cite{ZhangYang2008,WangYangYang2014} or on some
assumptions on optimal or suboptimal solutions of a perturbed problem
\cite{ShapiroSun2004,RuckmannShapiro,KanSong,KanSong2} that are very hard to verify. However, in some
particular cases there exist more constructive \textit{sufficient} conditions for the existence of augmented Lagrange
multipliers that rely on the use of sufficient optimality conditions and some compactness assumptions. These conditions
were obtained for mathematical programming problems
\cite{LiuTangYang,WangZhouXu,LuoMastroeniWu,ZhouXiuWang,WangLiuQu}, second-order cone programming problems
\cite{ZhouChen2015}, cone constrained optimization problems \cite{ZhouZhouYang2014}, nonlinear semidefinite programming
\cite{WuLuoYang}, and generalized semi-infinite min-max problems \cite{WangZhouXu}.

The main goal of this article is to develop two general approaches to the problem of existence of augmented Lagrange
multipliers that allow one to obtain simple and easily verifiable \textit{necessary and sufficient} conditions for 
the existence of augmented Lagrange multipliers. The first approach is based on the use of the theory of
exact penalty functions. Namely, we introduce a penalty function associated with the augmented Lagrangian, and prove
that an augmented Lagrange multiplier exists if and only if this penalty function is exact, provided that 
the augmenting function satisfies a certain growth condition. This result allows one to use the well-developed theory of
exactness of penalty functions \cite{Dolgopolik} in order to prove the existence of augmented Lagrange multipliers. In
particular, we prove that an augmented Lagrange multiplier of the sharp Lagrangian exists if and only if the standard
$\ell_1$ penalty function is exact. Furthermore, we demonstrate that if there exists an augmented Lagrange multiplier
of the sharp Lagrangian, then \textit{any} multiplier is an augmented Lagrange multiplier for this Lagrangian.

The second approach to the existence problem for augmented Lagrange multipliers that we develop in this article is
called \textit{the localization principle}. The localization principle allows one to study the \textit{local} behaviour
of the augmented Lagrangian near globally optimal solutions of the primal problem in order to prove the existence of an
augmented Lagrange multiplier. In turn, the local analysis of the augmented Lagrangian can be easily performed with the
use of sufficient optimality conditions. Thus, the localization principle reduces the problem of existence of augmented
Lagrange multipliers to the study of optimality conditions. The localization principle in an abstract form was first
formulated in the context of the theory of exact penalty function in \cite{Dolgopolik} (see~\cite{Dolgopolik},
Theorems~3.7, 3.10, 3.17, 3.20 and 3.21). In this paper, we demonstrate that the localization principle can be easily
extended to the theory of augmented Lagrangian functions. It should be noted that the localization principle allows one
to understand a general principle behind many known results on existence of augmented Lagrange multipliers (or saddle
points) of augmented Lagrangian functions (see
\cite{LiuTangYang,LuoMastroeniWu,ZhouXiuWang,WangLiuQu,ZhouChen2015,ZhouZhouYang2014,WuLuoYang,WangZhouXu},
and Remark~\ref{Rmrk_GeneralPrinciple} below).

The paper is organised as follows. The definition of augmented Lagrange multiplier and its useful reformulations are
given in Section~\ref{Sect_AugmLagrMultipliers}. In Section~\ref{Sect_ReductionToExPenFunc}, we introduce a penalty
function associated with the augmented Lagrangian, and study a connection between exactness of this penalty function and
the existence of augmented Lagrange multipliers. Section~\ref{Sect_LocalizationPrinciple} is devoted to the localization
principle, while in Section~\ref{Sect_Applications}, we present applications of the localization principle to 
mathematical programming and nonlinear semidefinite programming problems.

\section{Augmented Lagrange multipliers}
\label{Sect_AugmLagrMultipliers}

Let $X$ be a topological space, $A \subset X$ be a nonempty set, and $f \colon X \to \mathbb{R} \cup \{ + \infty \}$ be
a given function. Hereinafter, we study the optimization problem
$$
  \min f(x) \quad \text{subject to} \quad x \in A.	\eqno{(\mathcal{P})}
$$
We suppose that $\{ x \in A \mid f(x) < + \infty \} \ne \emptyset$, and that the optimal value 
$f^* = \inf_{x \in A} f(x)$ of the problem $(\mathcal{P})$ is finite.

The set $A$ represents constraints of the original problem that are not included into an augmented Lagrangian function.
In particular, the constraint $x \in A$ can represent simple constraints (e.g. bound or linear constraints).
It is clear that one can remove this constraint by defining $f(x) = + \infty$ for any $x \notin A$. However,
in the author's opinion, it is more convenient to include this constraint explicitly, since it allows one to
better understand how additional constraints affect the behaviour of the augmented Lagrangian function.

Denote $\overline{\mathbb{R}} = \mathbb{R} \cup \{ + \infty, - \infty \}$ and $\mathbb{R}_+ = [0, + \infty)$.
Let $P$ be a topological vector space of parameters. Recall that a function 
$\Phi \colon X \times P \to \overline{\mathbb{R}}$ is called a \textit{dualizing parameterization function} for $f$ if
$f(x) = \Phi(x, 0)$ for all $x \in X$. A function $\sigma \colon P \to [0, + \infty]$ is called an \textit{augmenting
function} if $\sigma(0) = 0$ and $\sigma(p) > 0$ for any $p \in P \setminus \{ 0 \}$. Below, we suppose that a dualizing
parameterization $\Phi$ for $f$, and an augmenting function $\sigma$ are fixed.

For any $p \in P$ denote by $v(p) = \inf_{x \in A} \Phi(x, p)$ \textit{the optimal value function} (or \textit{the
perturbation function}) of the problem $(\mathcal{P})$ associated with the dualizing parameterization $\Phi$. Note that
$v(0) = f^* > - \infty$.

The following definition of augmented Lagrangian function is a simple generalizations of the one given in
\cite{RockWets}.

\begin{definition}
Let $\Lambda$ be a vector space of \textit{multipliers}, and let the pair $(\Lambda, P)$ be equipped with 
a bilinear coupling function $\langle \cdot, \cdot \rangle \colon \Lambda \times P \to \mathbb{R}$.
\textit{The augmented Lagrangian with penalty parameter} $r \ge 0$ is the function 
$\mathscr{L} \colon X \times \Lambda \times \mathbb{R}_+ \to \overline{\mathbb{R}}$ defined by
$$
  \mathscr{L}(x, \lambda, r) = \inf_{p \in P} \Big( \Phi(x, p) - \langle \lambda, p \rangle + r \sigma(p) \Big).
$$
The corresponding \textit{augmented dual problem} consists of maximizing over all 
$(\lambda, r) \in \Lambda \times \mathbb{R}_+$ the function 
$\psi(\lambda, r) = \inf_{x \in A} \mathscr{L}(x, \lambda, r)$.
\end{definition}

Observe that for any $\lambda \in \Lambda$ and $r \ge 0$ one has $f^* \ge \psi(\lambda, r)$, i.e. the weak duality
holds. Furthermore, under some additional assumption (see \cite{HuangYang2003}, Theorem~2.1, and
\cite{HuangYang2005}, Theorem~2.1) one can show that the zero duality gap property holds true for
$\mathscr{L}(x, \lambda, r)$, i.e. $f^* = \sup\{ \psi(\lambda, r) \mid \lambda \in \Lambda, r \in \mathbb{R}_+ \}$,
if and only if the optimal value function $v$ is lower semicontinuous (l.s.c.) at the origin, and the optimal value of
the augmented dual problem is finite.

Let us also recall the definition of augmented Lagrange multiplier (see~\cite{ShapiroSun2004}).

\begin{definition}
A multiplier $\lambda \in \Lambda$ is called an \textit{augmented Lagrange multiplier} of the problem $(\mathcal{P})$
if there exists $r \ge 0$ such that
\begin{equation} \label{AugmLagrMult_Def}
  v(p) \ge v(0) + \langle \lambda, p \rangle - r \sigma(p) \quad \forall p \in P.
\end{equation}
The greatest lower bound of all such $r \ge 0$ is denoted by $r(\lambda)$, and is referred to as \textit{the least exact
penalty parameter} for $\lambda$. The set of all augmented Lagrange multipliers of $(\mathcal{P})$ is denoted by
$\mathcal{A}(\mathcal{P})$.
\end{definition}

Thus, if $\lambda \in \Lambda$ is an augmented Lagrange multiplier of $(\mathcal{P})$, then for any $r \ge r(\lambda)$
the inequality (\ref{AugmLagrMult_Def}) is valid, while for any $0 \le r < r(\lambda)$ there exists $p \in P$ such
that $v(p) < v(0) + \langle \lambda, p \rangle - r \sigma(p)$.

Augmented Lagrange multipliers of the problem $(\mathcal{P})$ can be described in terms of optimal solutions of the
augmented dual problem, and in terms of saddle points of the augmented Lagrangian. The propositions below are well-known
in the theory of augmented Lagrangian functions (cf.~\cite{RockWets}, Theorem~11.59, and \cite{ShapiroSun2004},
Theorem~2.1), and follow directly from definitions. Therefore we omit their proofs.

\begin{proposition}
If $\lambda^*$ is an augmented Lagrange multiplier of $(\mathcal{P})$, then the zero duality gap property holds true
for $\mathscr{L}(x, \lambda, r)$, and for any $r \ge r(\lambda^*)$ the pair $(\lambda^*, r)$ is a globally optimal
solution of the augmented dual problem. Conversely, if the zero duality gap property holds true 
for $\mathscr{L}(x, \lambda, r)$, then for any globally optimal solution $(\lambda^*, r^*)$ of the
augmented dual problem the vector $\lambda^*$ is an augmented Lagrange multiplier of $(\mathcal{P})$ and 
$r^* \ge r(\lambda^*)$.
\end{proposition}

\begin{proposition} \label{Prp_AugmLagr_GSP}
If $\lambda^*$ is an augmented Lagrange multiplier of $(\mathcal{P})$, then for all $r \ge r(\lambda^*)$ and for any
globally optimal solution $x^*$ of the problem $(\mathcal{P})$ the pair $(x^*, \lambda^*)$ is a global saddle point of
the function $\mathscr{L}(\cdot, \cdot, r)$, i.e.
$$
  \sup_{\lambda \in \Lambda} \mathscr{L}(x^*, \lambda, r) = \mathscr{L}(x^*, \lambda^*, r) = 
  \inf_{x \in A} \mathscr{L}(x, \lambda^*, r).
$$
Conversely, if a pair $(x^*, \lambda^*)$ is a global saddle point of the function $\mathscr{L}(\cdot, \cdot, r)$ for
some $r \ge 0$ and $\mathscr{L}(x^*, \lambda^*, r) = f^*$, then $\lambda^*$ is an augmented Lagrange multiplier of 
the problem $(\mathcal{P})$.
\end{proposition}

Recall that by definition a vector $\lambda \in \Lambda$ is an augmented Lagrange multiplier of $(\mathcal{P})$ if and
only if the inequality $v(p) \ge v(0) + \langle \lambda, p \rangle - r \sigma(p)$ is valid \textit{for all} $p \in P$.
However, under a natural additional assumption on the augmenting function $\sigma$, it is sufficient to verify that
this inequality is valid only for all $p$ in a neighbourhood of zero.

\begin{definition}
The augmenting function $\sigma$ is said to \textit{have a valley at zero} if for any neighbourhood $U \subset P$ of
zero there exists $\delta > 0$ such that $\sigma(p) \ge \delta$ for all $p \in P \setminus U$.
\end{definition}

The above definition naturally arose in the theory of augmented Lagrangian functions as a replacement of the convexity
assumption on the augmenting function $\sigma$ in \cite{RockWets}, and was utilized in many papers on
this subject (see, e.g. \cite{BurachikRubinov,ZhouYang2009,ZhouYang2012,ZhouZhouYang2014}). Note that in \cite{Penot},
functions similar to the ones having valley at zero were called \textit{potentials}. 

The proposition below, surprisingly, slightly improves all similar results on augmented Lagrange multipliers, since it
does not rely on any continuity and level-boundedness (coercivity) assumptions as well as any assumptions on the optimal
value function $v$ (cf.~\cite{RockWets}, Theorem~11.61; \cite{HuangYang2003}, Theorem~3.1; \cite{ShapiroSun2004},
Lemma~3.1, etc.).

\begin{proposition} \label{Prp_AugmLagrMult_NeighbZero}
Suppose that the augmenting function $\sigma$ has a valley at zero. Then $\lambda \in \Lambda$ is an augmented Lagrange
multiplier of $(\mathcal{P})$ if and only if there exist $r \ge 0$ and a neighbourhood $U \subset P$ of zero such that
\begin{equation} \label{AugmLagrMult_NeighbZero}
  v(p) \ge v(0) + \langle \lambda, p \rangle - r \sigma(p) \quad \forall p \in U ,
\end{equation}
and the function $\mathscr{L}(\cdot, \lambda, r)$ is bounded below on $A$.
\end{proposition}

\begin{proof}
If $\lambda \in \mathcal{A}(\mathcal{P})$, then (\ref{AugmLagrMult_NeighbZero}) is satisfied with $U = P$. Moreover,
for any $x \in A$ one has
$$
  \Phi(x, p) - \langle \lambda, p \rangle + r \sigma(p) \ge v(0) = f^* > - \infty \quad \forall p \in P.
$$
Taking the infimum over all $p \in P$ one obtains that $\mathscr{L}(x, \lambda, r) \ge f^*$ for any $x \in A$, i.e. 
$\mathscr{L}(\cdot, \lambda, r)$ is bounded below on $A$.

Suppose, now, that there exist $r \ge 0$ and a neighbourhood $U \subset P$ such that (\ref{AugmLagrMult_NeighbZero}) is
valid, and the function $\mathscr{L}(\cdot, \lambda, r)$ is bounded below on $A$. Then there exists $c > 0$ such
that for any $p \in P$ one has
$$
  \Phi(x, p) - \langle \lambda, p \rangle + r \sigma(p) \ge \mathscr{L}(x, \lambda ,r) \ge c \quad \forall x \in A.
$$
Taking the infimum over all $x \in A$ one gets that
$$
  v(p) - \langle \lambda, p \rangle + r \sigma(p) \ge c \quad \forall p \in P.
$$
By the definition of valley at zero there exists $\delta > 0$ such that $\sigma(p) \ge \delta$ for any 
$p \in P \setminus U$. Consequently, for any $p \in P \setminus U$ and for any 
$\tau > \overline{\tau} := r + (v(0) - c) / \delta$ one has
\begin{multline*}
  v(p) - \langle \lambda, p \rangle + \tau \sigma(p) = 
  v(p) - \langle \lambda, p \rangle + r \sigma(p) + (\tau - r) \sigma(p) \ge \\
  \ge v(p) - \langle \lambda, p \rangle + r \sigma(p) + (\tau - r ) \delta \ge c + (\tau - r) \delta \ge v(0).
\end{multline*}
Hence and from (\ref{AugmLagrMult_NeighbZero}) it follows that
$$
  v(p) \ge v(0) + \langle \lambda, p \rangle - \tau \sigma(p) \quad \forall p \in P \quad
  \forall \tau \ge \max\{ \overline{\tau}, r \},
$$
which implies that $\lambda \in \mathcal{A}(\mathcal{P})$.	 
\end{proof}

\begin{remark}
The proposition above can be easily extended to the general case. Namely, for any $\delta > 0$ denote
$K_{\delta} = \{ p \in P \mid \sigma(p) < \delta \}$. Then it is easy to check that 
$\lambda \in \mathcal{A}(\mathcal{P})$ if and only if there exist $r \ge 0$ and $\delta > 0$ such that
$$
  v(p) \ge v(0) + \langle \lambda, p \rangle - r \sigma(p) \quad \forall p \in K_{\delta},
$$
and the function $\mathscr{L}(\cdot, \lambda, r)$ is bounded below on $A$. Note that in the case when $\sigma$ has
a valley at zero, for any neighbourhood $U \subset P$ of zero there exists $\delta > 0$ such that 
$K_{\delta} \subset U$.
\end{remark}

Augmented Lagrange multipliers can be also characterized as those multipliers that \textit{support an exact penalty
representation} for the problem $(\mathcal{P})$ (see~\cite{RockWets}). 

\begin{proposition} \label{Prp_ExactPenaltyRepr}
For $\lambda \in \Lambda$ to be an augmented Lagrange multiplier of $(\mathcal{P})$ it is necessary and sufficient that
$\inf_{x \in A} \mathscr{L}(x, \lambda, r) = f^*$ for some $r \ge 0$. Moreover, the greatest lower bound of all such 
$r \ge 0$ coincides with $r(\lambda)$.
\end{proposition}

\begin{proof}
At first, note that $\inf_{x \in A} \mathscr{L}(x, \lambda, r) \le f^*$ for all $\lambda \in \Lambda$ due to the fact
that $\mathscr{L}(x, \lambda, r) \le \Phi(x, 0) - \langle \lambda, 0 \rangle + r \sigma(0) = f(x)$ 
for any $\lambda \in \Lambda$ and $x \in A$.

By definition, $\lambda \in \mathcal{A}(\mathcal{P})$ iff for some $r \ge 0$ one has
\begin{equation} \label{AugmLagrMult_InequalFromDef}
  v(p) \ge v(0) + \langle \lambda, p \rangle - r \sigma(p) \quad \forall p \in P.
\end{equation}
This inequality, in turn, is satisfied iff 
$$
  \Phi(x, p) - \langle \lambda, p \rangle + r \sigma(p) \ge v(0) = f^* \quad \forall (x, p) \in A \times P.
$$
Taking the infimum over all $p \in P$, and then over all $x \in A$ one obtains that (\ref{AugmLagrMult_InequalFromDef})
is valid for some $r \ge 0$ iff $\inf_{x \in A} \mathscr{L}(x, \lambda, r) \ge f^*$, which implies the desired result.
 
\end{proof}

\begin{remark}
Note that from the proposition above it follows that if $\mathcal{A}(\mathcal{P}) \ne \emptyset$, then the zero duality
gap property holds true for $\mathscr{L}(x, \lambda, r)$.
\end{remark}

\begin{corollary} \label{Crlr_ExactPenRepresentation}
Let $\lambda \in \Lambda$ be an augmented Lagrange multiplier of $(\mathcal{P})$. Then
$$
  \argmin_{x \in A} f(x) \subseteq \argmin_{x \in A} \mathscr{L}(x, \lambda, r) \quad \forall r \ge r(\lambda).
$$
Moreover, the equality
\begin{equation} \label{GlobMinEquiv}
  \argmin_{x \in A} f(x) = \argmin_{x \in A} \mathscr{L}(x, \lambda, r) \quad \forall r \ge r_0;
\end{equation}
holds true for some $r_0 \ge 0$ if and only if there exists $r_0 > r(\lambda)$ such that for any 
$x^* \in \argmin_{x \in A} \mathscr{L}(x, \lambda, r_0)$ one has
\begin{equation} \label{ExistsOfMinimizers}
  \argmin_{p \in P} \Big( \Phi(x^*, p) - \langle \lambda, p \rangle + r_0 \sigma(p) \Big) \ne \emptyset
\end{equation}
(i.e. the infimum in the definition of $\mathscr{L}(x^*, \lambda, r_0)$ is attained).
\end{corollary}

\begin{proof}
By the proposition above one has
\begin{equation} \label{InfOfAugmLagrFunc}
  \inf_{x \in A} \mathscr{L}(x, \lambda, r) = f^* \quad \forall r \ge r(\lambda).
\end{equation}
Let $x^*$ be a globally optimal solution of the problem $(\mathcal{P})$. Then
\begin{equation} \label{ParametrizationAtZeroAndOptSol}
  \Phi(x^*, 0) - \langle \lambda, 0 \rangle + r \sigma(0) = f(x^*) = f^* \quad \forall r \ge 0,
\end{equation}
which yields $\mathscr{L}(x^*, \lambda, r) \le f^*$. Therefore $x^*$ is a point of global minimum of 
$\mathscr{L}(\cdot, \lambda, r)$ on the set $A$ for any $r \ge r(\lambda)$.

Suppose, now, that \eqref{ExistsOfMinimizers} holds true for some $r_0 > r(\lambda)$, and $x^*$ is a global minimizer
of $\mathscr{L}(\cdot, \lambda, r_0)$ on $A$. Then there exists $p^* \in P$ such that
$$
  \mathscr{L}(x^*, \lambda, r_0) = \Phi(x^*, p^*) - \langle \lambda, p^* \rangle + r_0 \sigma(p^*).
$$
From \eqref{InfOfAugmLagrFunc} it follows that
$$
  \Phi(x, p) - \langle \lambda, p \rangle + r(\lambda) \sigma(p) \ge f^* \quad 
  \forall (x, p) \in A \times P \quad \forall r \ge r(\lambda).
$$
Recall that $\sigma(p) = 0$ iff $p = 0$. Therefore for any $r > r(\lambda)$ one has
$$
  \Phi(x, p) - \langle \lambda, p \rangle + r \sigma(p) > f^* \quad \forall x \in A \quad
  \forall p \in P \setminus \{ 0 \}.
$$
Hence taking into account \eqref{InfOfAugmLagrFunc} and the definitions of $x^*$ and $p^*$ one gets that $p^* = 0$ and
$\mathscr{L}(x^*, \lambda, r_0) = \Phi(x^*, 0) = f(x^*) = f^*$. Therefore $x^*$ is a globally optimal solution of the
problem $(\mathcal{P})$, i.e.
$$
  \argmin_{x \in A} f(x) = \argmin_{x \in A} \mathscr{L}(x, \lambda, r_0)
$$
Applying the fact that the function $\mathscr{L}(x, \lambda, r)$ is non-decreasing in $r$ one can easily verify that 
equality \eqref{GlobMinEquiv} is valid.		

Suppose, finally, that \eqref{GlobMinEquiv} holds true. Then taking into account \eqref{InfOfAugmLagrFunc} and
\eqref{ParametrizationAtZeroAndOptSol} one obtains that
for any $r \ge \max\{ r_0, r(\lambda) \}$ and 
$x^* \in \argmin_{x \in A} \mathscr{L}(x^*, \lambda, r)$ one has 
$\mathscr{L}(x^*, \lambda, r) = \Phi(x^*, 0) - \langle \lambda, 0 \rangle + r \sigma(0)$, i.e. $p = 0$ belongs to the
set on the left-hand side of \eqref{ExistsOfMinimizers}.  
\end{proof}

\begin{remark}
Note that the validity of equality \eqref{GlobMinEquiv} means that for any $r \ge r_0$ the problem $(\mathcal{P})$ is
equivalent (in terms of globally optimal solutions) to the problem of minimizing $\mathscr{L}(\cdot, \lambda, r)$ over 
the set~$A$.
\end{remark}

\section{Reduction to exact penalty functions}
\label{Sect_ReductionToExPenFunc}

In this section, we define a penalty function associated with the augmented Lagrangian $\mathscr{L}(x, \lambda, r)$, and
demonstrate how the exactness of this penalty function is connected with the existence of augmented Lagrange
multipliers of the problem $(\mathcal{P})$. Thus, the main results of this section allow one to use the well-developed
theory of exactness of penalty functions \cite{Dolgopolik} in order to prove the existence of augmented Lagrange
multipliers.

Choose a nonempty set $C \subseteq P$. For any $x \in X$ and $r \ge 0$ define
$$
  F(x, r, C) = \inf_{p \in C} \Big( \Phi(x, p) + r \sigma(p) \Big).
$$
If $C = P$, then we write $F(x, r)$ instead of $F(x, r, P)$. The function $F$ is called \textit{the penalty function}
associated with the augmented Lagrangian $\mathscr{L}(x, \lambda, r)$. 

The penalty function $F(x, r, C)$ is said to be \textit{exact} if there exists $r \ge 0$ such that $F(x, r, C) \ge f^*$
for all $x \in A$. It should be noted that this definition of exactness of a penalty function is equivalent to the
traditional one in the context of linear penalty functions (see, e.g.,~\cite{Dolgopolik}, Remark~8).

\begin{remark}
It is worth mentioning that $F(x, r) \equiv \mathscr{L}(x, 0, r)$, and the exactness of the penalty function $F(x, r)$
is equivalent to the fact that $\lambda = 0$ is an augmented Lagrange multiplier of $(\mathcal{P})$ 
(see Proposition~\ref{Prp_ExactPenaltyRepr})
\end{remark}

Let us present a simple example illuminating the notion of the penalty function associated with the augmented Lagrangian
function.

\begin{example} \label{Exmpl_AssociatedPenFunc}
Suppose that the initial optimization problem has the form
$$
  \min f_0(x) \quad \text{subject to} \quad 0 \in G(x), \quad x \in A,		\eqno{(\mathcal{M})}	
$$
where $f_0 \colon X \to \mathbb{R}$, and $G \colon X \rightrightarrows P$ is a set-valued mapping with closed values.
Define $f(x) = f_0(x)$, if $0 \in G(x)$, and $f(x) = + \infty$, otherwise. Then the problem $(\mathcal{P})$ is
equivalent to the problem $(\mathcal{M})$. Define
$$
  \Phi(x, p) = \begin{cases}
    f_0(x), & \text{if } 0 \in G(x) + p, \\
    + \infty, & \text{otherwise}
  \end{cases}
$$
(we call this dualizing parameterization \textit{standard}). Then for any $C \subset P$, $x \in X$ and $r \ge 0$ one has
$$
  F(x, r, C) = f_0(x) + r \inf\{ \sigma(p) \mid p \in (-G(x)) \cap C \}.
$$
In particular, if $P$ is a normed space, and $\sigma(p) = \omega(\| p \|)$, where 
$\omega \colon [0, + \infty] \to [0, + \infty]$ is a non-decreasing function, then
$$
  F(x, r, C) = f_0(x) + r \omega\Big( d\big( 0, (-G(x)) \cap C \big) \Big).
$$
In the case $\sigma(p) = \| p \|$, one gets that $F(x, r) = f_0(x) + r d(0, G(x))$. Thus, the standard penalty
function $F(x, r) = f_0(x) + r d(0, G(x))$ for $(\mathcal{M})$ is associated with the sharp Lagrangian
(see~\cite{RockWets}, Example~11.58). Similarly, the quadratic penalty function $F(x, r) = f_0(x) + (r/2) d(0, G(x))^2$
is associated with the proximal Lagrangian (i.e. the augmented Lagrangian with $\sigma(p) = \| p \|^2 / 2$).

Note also that if $C = \{ p \in P \mid \| p \| < \delta \}$ for some $\delta > 0$, then
$F(x, r, C) = f_0(x) + \omega(d(0, G(x))$, if $d(0, G(x)) < \delta$, and $F(x, r, C) = + \infty$, otherwise.
\end{example}

The penalty function $F(x, r)$ can be utilized to obtain a simple characterization of augmented Lagrange multipliers
of the problem $(\mathcal{P})$ in the case when $\mathscr{L}(x, \lambda, r)$ is the sharp Lagrangian.

\begin{proposition} \label{Prp_SharpLagrEquivExactPenalty}
Let $P$ be a normed space, and assume that for any $\lambda \in \Lambda$ there exists $d(\lambda) \ge 0$ such that
$|\langle \lambda, p \rangle| \le d(\lambda) \| p \|$ for all $p \in P$. Suppose also that 
$\sigma(\cdot) = \| \cdot \|$. Then an augmented Lagrange multiplier of $(\mathcal{P})$ exists if and only if the
penalty function $F(x, r)$ is exact. Moreover, if $\mathcal{A}(\mathcal{P}) \ne \emptyset$, then 
$\mathcal{A}(\mathcal{P}) = \Lambda$.
\end{proposition}

\begin{proof}
Suppose that $\mathcal{A}(\mathcal{P}) \ne \emptyset$ and $\lambda \in \mathcal{A}(\mathcal{P})$. Then by
Proposition~\ref{Prp_ExactPenaltyRepr} one has that
$\inf_{x \in A} \mathscr{L}(x, \lambda, r) = f^*$ for all $r \ge r(\lambda)$. Observe that for any $x \in X$ and 
$p \in P$ one has
$$
  \Phi(x, p) - \langle \lambda, p \rangle + r \sigma(p) \le
  \Phi(x, p) + (d(\lambda) + r) \| p \|.
$$
Taking the infimum over all $p \in P$ one obtains that
$$
  f^* \le \mathscr{L}(x, \lambda, r) \le F(x, d(\lambda) + r) \quad \forall x \in A \quad \forall r \ge r(\lambda),
$$
which implies that the penalty function $F(x, r)$ is exact.

Suppose, now, that $F(x, r)$ is exact. For any $\lambda \in \Lambda$ and $r \ge 0$ one has
$$
  \Phi(x, p) - \langle \lambda, p \rangle + r \sigma(p) \ge
  \Phi(x, p) + (r - d(\lambda)) \| p \| \quad \forall (x, p) \in X \times P,
$$
which implies that $\mathscr{L}(x, \lambda, r) \ge F(x, r - d(\lambda))$ for any $x \in X$ and $r \ge d(\lambda)$.
Therefore taking into account the fact that the penalty function $F(x, r)$ is exact one obtains that for any
sufficiently large $r$ one has that $\mathscr{L}(x, \lambda, r) \ge f^*$ for all $x \in A$. Hence with the use of
Proposition~\ref{Prp_ExactPenaltyRepr} one gets that $\lambda$ is an augmented Lagrange
multiplier of the problem $(\mathcal{P})$, and $\mathcal{A}(\mathcal{P}) = \Lambda$.	 
\end{proof}

\begin{remark}
Consider the problem $(\mathcal{M})$ from Example~\ref{Exmpl_AssociatedPenFunc}. Let $P$ be a normed space,
$\sigma(\cdot) = \| \cdot \|$, and let $\Phi$ be the standard dualizing parameterization for this problem.
Then $\mathscr{L}(x, \lambda, r)$ is the sharp Lagrangian.

From the proposition above it follows that an augmented Lagrange multiplier of the problem $(\mathcal{P})$ exists, in
the case when $\mathscr{L}$ is the sharp Lagrangian, if and only if the penalty function 
$F(x, r) = f_0(x) + r d(0, G(x))$ is exact. In particular, by \cite{Dolgopolik}, Remark~18 an augmented Lagrange
multiplier of $(\mathcal{P})$ exists iff the function $f_0(x) + r d(0, G(x))$ is bounded below on $A$ for some 
$r \ge 0$, and the optimal value function $v(p) = \inf\{ f_0(x) \mid x \in A \colon 0 \in G(x) + p \}$ is calm from
below at the origin. Moreover, by the previous proposition, if there exists an augmented Lagrange multiplier of the
problem $(\mathcal{P})$ in the case when $\mathscr{L}$ is the sharp Lagrangian, then \textit{any} multiplier 
$\lambda \in \Lambda$ is an augmented Lagrange multiplier of the problem $(\mathcal{P})$ (cf.
\cite{BurachikIusemMelo2015}, Corollary~3.8). This result, in particular, explains the rapid convergence (only few
iterations are needed to find a good approximation of a global minimizer) of the primal-dual method based on the use of
the sharp Lagrangian \cite{Gasimov,BurachikGasimovEtAl}.

Thus, the only potential benefit of the use of the sharp Lagrangian instead of the standard exact penalty function for
the problem $(\mathcal{M})$ is the smaller value of the least exact penalty parameter $r(\lambda)$ for some 
$\lambda \in \Lambda$ than the value of the least exact penalty parameter $r(0)$ of the penalty function. However, it
should be noted that for some problems one has $r(\lambda) > r(0)$ for any $\lambda \ne 0$. In particular, it is easy
to see that for the problem
\begin{equation} \label{SimpleExampleProblem}
  \min x_1^2 - |x_2| \quad \text{subject to} \quad x_2 = 0
\end{equation}
with $P = \Lambda = \mathbb{R}$ and $\sigma(\cdot) = | \cdot |$ one has 
$$
  \mathscr{L}(x_1, x_2, \lambda, r) = x_1^2 - |x_2| + \lambda x_2 + r |x_2| \quad 
  \forall \lambda \in \mathbb{R}, \: r \ge 0,
$$
and $r(\lambda) = 1 + |\lambda|$ for any $\lambda \in \Lambda$. Therefore $r(\lambda) > r(0)$ for any 
$\lambda \ne 0$, where $r(0)$ is, in fact, the least exact penalty parameter of the $\ell_1$ penalty function for the
problem \eqref{SimpleExampleProblem}.
\end{remark}

Further results on a connection between the existence of augmented Lagrange multipliers and the exactness of the
penalty function $F(x, r, C)$ can be obtained under a simple growth assumption on the augmenting function $\sigma$.
Namely, if $\sigma$ grows near the origin at least as fast as the norm $\| p \|$, then, roughly speaking, an augmented
Lagrange multiplier of $(\mathcal{P})$ exists iff the penalty function $F(x, r, C)$ is exact with $C$ being a
neighbourhood of zero.

\begin{proposition}
Let $P$ be a normed space, and let for any $\lambda \in \Lambda$ there exists $d(\lambda) \ge 0$ such that
$|\langle \lambda, p \rangle| \le d(\lambda) \| p \|$ for all $p \in P$. Suppose also that 
$\liminf_{p \to 0} \sigma(p) / \| p \| > 0$, and $\sigma$ has a valley at zero. Then an augmented Lagrange multiplier of
$(\mathcal{P})$ exists if and only if there exist a neighbourhood $U \subset P$ of zero, and $\lambda \in \Lambda$ such
that the penalty function $F(x, r, U)$ is exact, and the function  $\mathscr{L}(\cdot, \lambda, r)$ is bounded below on
$A$ for some $r \ge 0$. Furthermore, if $\mathcal{A}(\mathcal{P}) \ne \emptyset$, then $\mathcal{A}(\mathcal{P})$
consists of all those $\lambda \in \Lambda$ for which $\mathscr{L}(\cdot, \lambda, r)$ is bounded below on $A$ for some
$r \ge 0$.
\end{proposition}

\begin{proof}
Suppose that $\mathcal{A}(\mathcal{P}) \ne \emptyset$ and $\lambda \in \mathcal{A}(\mathcal{P})$. Then by
Proposition~\ref{Prp_ExactPenaltyRepr} one has that $\inf_{x \in A} \mathscr{L}(x, \lambda, r) = f^*$ for all 
$r \ge r(\lambda)$. From the definition of limit inferior it follows that there exist a neighbourhood $U \subset P$ of
zero and $\sigma_0 > 0$ such that 
\begin{equation} \label{AugmentingFuncLowerEstimAtZero}
  \sigma(p) \ge \sigma_0 \| p \| \quad \forall p \in U.
\end{equation}
Therefore for any $x \in X$, $p \in U$ and $r \ge 0$ one has
$$
  \Phi(x, p) - \langle \lambda, p \rangle + r \sigma(p) \le
  \Phi(x, p) + \left( \frac{d(\lambda)}{\sigma_0} + r \right) \sigma(p).
$$
Consequently, taking the infimum over all $p \in U$ one obtains that for any $x \in A$ and $r \ge r(\lambda)$ the
following inequalities hold true
$$
  f^* \le \mathscr{L}(x, \lambda, r) \le 
  \inf_{p \in U} \Big( \Phi(x, p) - \langle \lambda, p \rangle + r \sigma(p) \Big) \le 
  F\left( x, \frac{d(\lambda)}{\sigma_0} + r, U \right),
$$
which implies that the penalty function $F(x, r, U)$ is exact.

Suppose, now, that there exist a neighbourhood $U \subset P$ of zero and $\lambda \in \Lambda$ such that
the penalty function $F(x, r, U)$ is exact, and the function $\mathscr{L}(\cdot, \lambda, r_0)$ is bounded below on
$A$ for some $r_0 \ge 0$. By the definition of exactness of a penalty function for any sufficiently large $r \ge 0$
one has that $F(x, r, U) \ge f^*$ for any $x \in A$. Applying \eqref{AugmentingFuncLowerEstimAtZero} one gets that
$$
  \Phi(x, p) - \langle \lambda, p \rangle + r \sigma(p) \ge
  \Phi(x, p) + \left( r - \frac{d(\lambda)}{\sigma_0} \right) \sigma(p) \quad \forall (x, p) \in X \times P,
$$
which implies that for all $(x, p) \in A \times U$ and for any $r$ large enough the following inequalities hold true
$$
  \Phi(x, p) - \langle \lambda, p \rangle + r \sigma(p) \ge 
  F\left(x, r - \frac{d(\lambda)}{\sigma_0}, U \right) \ge f^*.
$$
Taking the infimum over all $x \in A$ one obtains that
$$
  v(p) - \langle \lambda, p \rangle + r \sigma(p) \ge f^* = v(0) \quad \forall p \in U.
$$
Hence and from Proposition~\ref{Prp_AugmLagrMult_NeighbZero} it follows that $\lambda$ is an augmented Lagrange
multiplier of the problem $(\mathcal{P})$.	 
\end{proof}

\begin{remark}
Consider the problem $(\mathcal{M})$ from Example~\ref{Exmpl_AssociatedPenFunc}. Let $P$ be a normed space,
$\sigma(\cdot) = \| \cdot \|^{\gamma}$ with $\gamma \in (0, 1)$, and let $\Phi$ be the standard dualizing
parameterization for this problem. In this case we call $\mathscr{L}(x, \lambda, r)$ \textit{the lower order Lagrangian}
(cf.~\cite{WuBaiYang,BaiWuZhu}).

From the proposition above it follows that an augmented Lagrange multiplier of the problem ($\mathcal{P}$) exists in
the case when $\mathscr{L}$ is the lower order Lagrangian if and only if there exist $\tau > 0$ and 
$\lambda \in \Lambda$ such that the lower order penalty function
\begin{equation} \label{LowerOrderPenFunc}
  F_{\gamma}(x, r) = \begin{cases}
    f_0(x) + r d(0, G(x))^{\gamma}, & \text{if } d(0, G(x)) < \tau, \\
    + \infty, & \text{otherwise}
  \end{cases}
\end{equation}
is exact, and the function $\mathscr{L}(\cdot, \lambda, r)$ is bounded below on $A$ for some $r \ge 0$. Moreover,
if $\mathcal{A}(\mathcal{P}) \ne \emptyset$, then any multiplier $\lambda \in \Lambda$ such that 
$\mathscr{L}(\cdot, \lambda, r)$ is bounded below on $A$ for some $r \ge 0$ is an augmented Lagrange multiplier of
the problem $(\mathcal{P})$.

It should be noted that from \cite{Dolgopolik}, Theorems~3.22 and 3.23 it follows that the penalty function
(\ref{LowerOrderPenFunc}) is exact iff 
$$
  \liminf_{p \to 0} \frac{v(p) - v(0)}{\| p \|^{\gamma}} > - \infty,
$$
i.e. iff the optimal value function $v$ is lower order calm at the origin (see also~\cite{BaiWuZhu}). One can also
verify that the penalty function (\ref{LowerOrderPenFunc}) is exact iff there exists $a > 0$ such that the penalty
function
$$
  \widehat{F}_{\gamma}(x, r) = \begin{cases}
    f_0(x) + r \dfrac{d(0, G(x))^{\gamma}}{a - d(0, G(x))^{\gamma}}, & \text{if } d(0, G(x))^{\gamma} < a, \\
    + \infty, & \text{otherwise}.
  \end{cases}
$$
is exact. Note that the penalty function $\widehat{F}_{\gamma}(x, r)$ is more suitable for numerical optimization
than the penalty function \eqref{LowerOrderPenFunc}.
\end{remark}

If the augmenting function $\sigma$ grows near the origin slower than then the norm $\| p \|$, then the exactness of a
certain penalty function is necessary (but not sufficient) for the existence of an augmented Lagrange multiplier.

\begin{proposition}
Let $P$ be a normed space, and assume that for any $\lambda \in \Lambda$ there exists $d(\lambda) \ge 0$ such that
$|\langle \lambda, p \rangle| \le d(\lambda) \| p \|$ for all $p \in P$. Suppose also that 
$\limsup_{p \to 0} \sigma(p)/ \| p \| < + \infty$. Then for the existence of an augmented Lagrange multiplier of
the problem $(\mathcal{P})$ it is necessary that there exists a neighbourhood $U \subset P$ of zero such that the
penalty function $F_1(x, r, U) := \inf_{p \in U} ( \Phi(x, p) + r \| p \|)$ is exact.
\end{proposition}

\begin{proof}
From the definition of limit superior it follows that there exist a neighbourhood $U \subset P$ of zero and 
$\sigma_0 > 0$ such that $\sigma(p) \le \sigma_0 \| p \|$ for all $p \in U$. Hence for any $\lambda \in \Lambda$ and 
$r \ge 0$ one has
$$
  \Phi(x, p) - \langle \lambda, p \rangle + r \sigma(p) \le
  \Phi(x, p) + ( d(\lambda) + r \sigma_0 ) \| p \| \quad \forall (x, p) \in X \times P.
$$
Taking the infimum over all $p \in U$ one obtains that
$$
  F_1(x, d(\lambda) + r \sigma_0, U) \ge 
  \inf_{p \in U} \Big( \Phi(x, p) - \langle \lambda, p \rangle + r \sigma(p) \Big) \ge
  \mathscr{L}(x, \lambda, r) \quad \forall x \in X.
$$
Thus, if $\lambda$ is an augmented Lagrange multiplier of $(\mathcal{P})$ and $r \ge r(\lambda)$, then by
Proposition~\ref{Prp_ExactPenaltyRepr} one has $F_1(x, d(\lambda) + r \sigma_0, U) \ge f^*$, i.e. $F(x, r, U)$ is exact.
 
\end{proof}

\begin{remark}
Consider the problem $(\mathcal{M})$ from Example~\ref{Exmpl_AssociatedPenFunc}. Let $P$ be a normed space,
$\sigma(\cdot) = \| \cdot \|^{\beta}$ with $\beta > 1$, and let $\Phi$ be the standard dualizing
parameterization for this problem. From the proposition above it follows that, in this case, for the existence of an
augmented Lagrange multiplier of the problem $(\mathcal{P})$ it is necessary that there exists $\tau > 0$ such that the
penalty function \eqref{LowerOrderPenFunc} with $\gamma = 1$ is exact. In particular, from \cite{Dolgopolik}, Theorem
3.22 it follows that for the existence of an augmented Lagrange multiplier of the problem $(\mathcal{P})$ it is
necessary that the optimal value function $v$ is calm from below at the origin.
\end{remark}

Finally, note that if the function $\sigma$ grows near the origin ``too slow'' (namely, slower than $\| p \|^2$),
then, in the general case, augmented Lagrange multipliers of the problem $(\mathcal{P})$ do not exist.

\begin{example}
Let $X$ and $P$ be normed spaces, $\Lambda$ be the topological dual of $P$, and $\langle \cdot, \cdot \rangle$ be the
standard coupling function, i.e. $\langle \lambda, p \rangle = \lambda(p)$. Consider the following optimization
problem
\begin{equation} \label{MathProg_BanachSpace}
  \min f_0(x) \quad \text{subject to} \quad h(x) = 0,
\end{equation}
where $f_0 \colon X \to \mathbb{R}$ and $h \colon X \to P$ are given functions. Let $\Phi(x, p)$ be the standard
dualizing parameterization, and $\sigma(p) = \| p \|^{2 + \varepsilon}$, where $\varepsilon > 0$. Then
$$
  \mathscr{L}(x, \lambda, r) = f_0(x) + \langle \lambda, h(x) \rangle + r \| h(x) \|^{2 + \varepsilon}.
$$
Let $x^* \in X$ be a globally optimal solution of the problem \eqref{MathProg_BanachSpace}. Suppose that the functions 
$f$ and $h$ are continuously Fr\'echet differentiable at $x^*$, and there exist the second Fr\'echet derivatives of
these functions at the point $x^*$. Suppose, finally, that there exists an augmented Lagrange multiplier 
$\lambda_0 \in \Lambda$ of the problem $(\mathcal{P})$. Then by Corollary~\ref{Crlr_ExactPenRepresentation} the point
$x^*$ is a global minimizer of $\mathscr{L}(\cdot, \lambda_0, r)$ for any $r \ge r(\lambda_0)$.

From the fact that $h$ is continuously Fr\'echet differentiable at $x^*$ it follows that $h$ is Lipschitz
continuous near $x^*$. Note that $h(x^*) = 0$ due to the fact that $x^*$ is an optimal solution of 
the problem \eqref{MathProg_BanachSpace}. Therefore there exists $L > 0$ such that for any $x$ in a neighbourhood of
$x^*$ one has $\| h(x) \| \le L \| x \|$. Hence 
$\omega(x) := \| h(x) \|^{2 + \varepsilon} \le L^{2 + \varepsilon} \| x \|^{2 + \varepsilon}$ for any $x$
sufficiently close to $x^*$. Consequently, the function $\omega$ is twice Fr\'echet differentiable at $x^*$,
$\omega'(x^*) = 0$ and $\omega''(x^*) = 0$. Therefore the augmented Lagrangian $\mathscr{L}(\cdot, \lambda_0, r)$ is
twice Fr\'echet differentiable at $x^*$. By the necessary condition for a minimum one has
$$
  D^2_{xx} \mathscr{L}(x^*, \lambda_0, r) (y, y) \ge 0 \quad \forall y \in X \quad \forall r \ge r(\lambda_0).
$$
Taking into account the fact that $\omega''(x^*) = 0$ one obtains that
$$
  D^2_{xx} \mathscr{L}(x^*, \lambda_0, r) (y, y) = D^2_{xx} L(x^*, \lambda_0) (y, y) \ge 0 \quad \forall y \in X,
$$
where $L(x, \lambda_0) = f_0(x) + \langle \lambda, h(x) \rangle$ is the standard Lagrangian for the problem
\eqref{MathProg_BanachSpace}. However, usually, there is no such $\lambda \in \Lambda$ that
$D^2_{xx} L(x^*, \lambda) (y, y) \ge 0$ for all $y \in X$, unless $x^*$ is an unconstrained minimum of $f_0$. 

For instance, let the problem \eqref{MathProg_BanachSpace} have the form
\begin{equation} \label{SimpleProbl_SaddlePoint}
  \min x_1^2 - x_2 + 2 \cos x_2 \quad \text{subject to} \quad x_2 = 0.
\end{equation}
Then $L(x, \lambda) = x_1^2 - x_2 + 2 \cos x_2 + \lambda x_2$, $x^* = (0, 0)$, and for any $\lambda \in \mathbb{R}$ one
has
$$
  D^2_{xx} L(x^*, \lambda) (y, y) = 2 y_1^2 - 2 y_2^2 < 0 \quad \forall y \in \mathbb{R}^2 \colon |y_2| > |y_1|.
$$
Thus, in the general case, an augmented Lagrange multiplier of the problem $(\mathcal{P})$ does not exist when 
$\sigma(p) = \| p \|^{2 + \varepsilon}$ with $\varepsilon > 0$. However, note that $\lambda^* = 1$ is
an augmented Lagrange multiplier of the problem \eqref{SimpleProbl_SaddlePoint}, and $r(\lambda^*) = 4$ in the case when
$\mathscr{L}(x, \lambda, r)$ is the proximal Lagrangian, i.e. when $\sigma(p) = \| p \|^2 / 2$.
\end{example}

\section{Localization principle}
\label{Sect_LocalizationPrinciple}

In this section, we describe a general method for proving the existence of augmented Lagrange multipliers in the finite
dimensional case called \textit{the localization principle}. Clearly, the existence of an augmented Lagrange multiplier
is a \textit{global} property of the augmented Lagrangian function. However, the localization principle allows one to
study \textit{local} behaviour of an augmented Lagrangian near globally optimal solution of the problem $(\mathcal{P})$
in order to prove the existence of an augmented Lagrange multiplier. In turn, local analysis of an augmented Lagrangian
is usually performed with the use of sufficient optimality conditions (see Section~\ref{Sect_Applications} below). Thus,
the localization principle reduces the problem of existence of augmented Lagrange multipliers to the study of
optimality conditions. 

Before we turn to the localization principle, we need to obtain an auxiliary result on minimizing sequences constructed
with the use of an augmented Lagrangian.

\begin{proposition} \label{Prp_MinSeqAugmLagr}
Let $\{ r_n \} \subset (0, + \infty)$ be an increasing unbounded sequence, and let a sequence 
$\{ \varepsilon_n \} \subset (0, + \infty)$ be such that $\varepsilon_n \to 0$ as $n \to \infty$. 
Let also a multiplier $\lambda \in \Lambda$ be fixed. Suppose that the following assumptions are satisfied:
\begin{enumerate}
\item{$A$ is closed;}

\item{$\Phi$ is l.s.c. on $A \times \{ 0 \}$;}

\item{the function $\langle \lambda, \cdot \rangle$ is continuous at the origin;}

\item{$\sigma$ has a valley at zero;}

\item{the function $\mathscr{L}(\cdot, \lambda, r_1)$ is bounded below on $A$.}
\end{enumerate}
Then any cluster point of a sequence $\{ x_n \} \subset A$ such that
$$
  \mathscr{L}(x_n, \lambda, r_n) \le \inf_{x \in A} \mathscr{L}(x, \lambda, r_n) + \varepsilon_n 
  \quad \forall n \in \mathbb{N}.
$$
is a globally optimal solution of the problem $(\mathcal{P})$.
\end{proposition}

\begin{proof}
By the definition of the augmented Lagrangian $\mathscr{L}(x, \lambda, r)$ for any $n \in \mathbb{N}$ there exists 
$p_n \in \mathbb{N}$ such that
$$
  \Phi(x_n, p_n) - \langle \lambda, p_n \rangle + r_n \sigma(p_n) \le \mathscr{L}(x_n, \lambda, r_n) + \varepsilon_n.
$$
Observe that $\inf_{x \in A} \mathscr{L}(x, \lambda, r) \le f^*$ for any $r \ge 0$ due to the fact that 
$\Phi(x, 0) = f(x)$ for all $x \in X$. Therefore for any $n \in \mathbb{N}$ one has
\begin{multline} \label{MinSeqLowerEstimate}
  \varepsilon_n + f^* \ge \mathscr{L}(x_n, \lambda, r_n) \ge 
  \Phi(x_n, p_n) - \langle \lambda, p_n \rangle + r_1 \sigma(p_n) + (r_n - r_1) \sigma(p_n) - \varepsilon_n \ge \\
  \ge \mathscr{L}(x_n, \lambda, r_1) + (r_n - r_1) \sigma(p_n) - \varepsilon_n \ge
  c + (r_n - r_1) \sigma(p_n) - \varepsilon_n,
\end{multline}
where $c = \inf_{x \in A} \mathscr{L}(x, \lambda, r_1) > - \infty$. Consequently, taking into account the fact that
$\varepsilon_n \to 0$ as $n \to \infty$, and $r_n$ is an increasing unbounded sequence one obtains that 
$\sigma(p_n) \to 0$ as $n \to \infty$. Hence $p_n \to 0$ as $n \to \infty$ due to the fact that
$\sigma$ has a valley at zero.

Let $x^*$ be a cluster point of the sequence $\{ x_n \}$. Note that $x^* \in A$ by virtue of the fact that $A$ is
closed. From \eqref{MinSeqLowerEstimate} it follows that
$$
  \Phi(x_n, p_n) - \langle \lambda, p_n \rangle + r_1 \sigma(p_n) \le f^* + 2 \varepsilon_n.
$$
Passing to the limit inferior as $n \to \infty$, and taking into account the assumptions on the functions $\Phi$ and
$\langle \lambda, \cdot \rangle$, and the fact that $\sigma(p_n) \to 0$ as $n \to \infty$ one obtains that
$f(x^*) = \Phi(x^*, 0) \le f^*$, which implies that $x^*$ is globally optimal solution of the problem $(\mathcal{P})$.
 
\end{proof}

\begin{remark} \label{Rmrk_MinNetAugmLagr}
Suppose that all assumptions of the proposition above are satisfied. Let also for any $r \ge r_1$ a point $x(r) \in A$
be such that
$$
  \mathscr{L}(x(r), \lambda, r) \le \inf_{x \in A} \mathscr{L}(x, \lambda, r) + \varepsilon(r),
$$
where $\varepsilon(r) > 0$ and $\varepsilon(r) \to 0$ as $r \to \infty$. Then arguing in the same way as in the proof
of the previous proposition one can verify that any cluster point of the net $\{ x(r) \}$, $r \in [r_1, + \infty)$, if
exists, is a globally optimal solution of the problem $(\mathcal{P})$.
\end{remark}

As it was mentioned above, the localization principle allows one to study local behaviour of the augmented
Lagrangian near globally optimal solutions of the problem $(\mathcal{P})$ in order to prove the existence of an
augmented
Lagrange multiplier. The following definition describes the desired local behaviour of an augmented Lagrangian.

\begin{definition}
Let $x^*$ be a locally optimal solution of the problem $(\mathcal{P})$. A multiplier $\lambda \in \Lambda$ is called a
(\textit{local}) \textit{augmented Lagrange multiplier} at the point $x^*$ if there exists $r_0 \ge 0$ such that for
any $r \ge r_0$ the point $x^*$ is a local minimizer of the function $\mathscr{L}(\cdot, \lambda, r)$ on the set $A$,
and $\mathscr{L}(x^*, \lambda, r) = f(x^*)$. The greatest lower bound of all such $r_0$ is denoted by
$r(x^*, \lambda)$ and is called \textit{the least exact penalty parameter} for the multiplier $\lambda$ at the point
$x^*$. The set of all local augmented Lagrange multipliers at the point $x^*$ is denoted by $\mathcal{A}(x^*)$.
\end{definition}

Let $\lambda^*$ be a local augmented Lagrange multiplier at a locally optimal solution $x^*$ of the problem
$(\mathcal{P})$. From the fact that $\Phi(x^*, 0) = f(x^*)$ it follows that for any $\lambda \in \Lambda$ and $r \ge 0$
one has $\mathscr{L}(x^*, \lambda, r) \le f(x^*)$. Hence and from the definition above it follows that there exist a
neighbourhood $U$ of $x^*$ and $r_0 \ge 0$ such that
$$
  \sup_{\lambda \in \Lambda} \mathscr{L}(x^*, \lambda, r) \le \mathscr{L}(x^*, \lambda^*, r) \le
  \inf_{x \in U \cap A} \mathscr{L}(x, \lambda^*, r) \quad \forall r \ge r_0,
$$
i.e. $(x^*, \lambda^*)$ is a local saddle point of the augmented Lagrangian $\mathscr{L}(x, \lambda, r)$. On the other
hand, if $(x^*, \lambda^*)$ is a local saddle point of $\mathscr{L}(x, \lambda, r)$ such that 
$\mathscr{L}(x^*, \lambda^*, r) = f(x^*)$ for some $r \ge 0$, then, as it is easy to see, the multiplier $\lambda^*$ is
a local augmented Lagrange multiplier at $x^*$. Thus, there is a close connection between local augmented Lagrange
multipliers and local saddle points of the augmented Lagrangian $\mathscr{L}(x, \lambda, c)$ 
(cf.~Proposition~\ref{Prp_AugmLagr_GSP}).

Denote
$$
  \mathcal{A}_{loc}(\mathcal{P}) = \bigcap \mathcal{A}(x^*),
$$
where the intersection is taken over all globally optimal solutions $x^*$ of the problem $(\mathcal{P})$. From
Corollary~\ref{Crlr_ExactPenRepresentation} it follows that 
$\mathcal{A}(\mathcal{P}) \subset \mathcal{A}_{loc}(\mathcal{P})$, provided the problem $(\mathcal{P})$ has
optimal solutions. As the following example shows, the opposite inclusion does not hold true in the general case.

\begin{example}
Let $X = A = P = \Lambda = \mathbb{R}$ and $\langle \lambda, p \rangle = \lambda p$ for all 
$(\lambda, p) \in \Lambda \times P$. Let $f_0(x) = - x$ for all $x \le 0$, $f_0(x) = - x^2$ for all $x > 0$,
and let the problem under consideration have the form
\begin{equation} \label{ExampleUnboundedLagrangian}
  \min f_0(x) \quad \text{subject to} \quad x \le 0.
\end{equation}
Define $f(x) = f_0(x)$ for any $x \le 0$, and $f(x) = + \infty$, otherwise. Then 
the problem \eqref{ExampleUnboundedLagrangian} is equivalent to the problem $(\mathcal{P})$. Note that the point 
$x^* = 0$ is a unique globally optimal solution of the problem \eqref{ExampleUnboundedLagrangian} and $f^* = 0$.

Define $\sigma(p) = |p|$, and let $\Phi(x, p)$ be the standard dualizing parameterization for this problem, i.e.
$\Phi(x, p) = f_0(x)$, if $x + p \le 0$, and $\Phi(x, p) = + \infty$, otherwise. Then one can easily check that for all
$x, \lambda \in \mathbb{R}$ and $r \ge 0$ one has
\begin{equation} \label{Example_LocalNeEqGlobalALM}
  \mathscr{L}(x, \lambda, r) = \begin{cases}
    f_0(x) + (r + \lambda) \max\{ x, 0 \}, & \text{if~} r \ge |\lambda|, \\
    - \infty, & \text{if~} r < |\lambda|, \: \lambda < 0, \\
    f_0(x) + \lambda x + r |x|, & \text{if~} r < |\lambda|, \: \lambda > 0.
  \end{cases}
\end{equation}
Therefore $\inf_{x \in \mathbb{R}} \mathscr{L}(x, \lambda, r) = - \infty$ for any $\lambda \in \mathbb{R}$ and 
$r \ge 0$, which implies that $\mathcal{A}(\mathcal{P}) = \emptyset$. However, it is easy to verify that the point
$x^* = 0$ is a local minimizer of $\mathscr{L}(x, \lambda, r)$, provided $r > |\lambda|$. Thus,
$\mathcal{A}_{loc}(\mathcal{P}) = \mathbb{R}$ and $\mathcal{A}(\mathcal{P}) \ne \mathcal{A}_{loc}(\mathcal{P})$.
Note that the augmented Lagrangian \eqref{Example_LocalNeEqGlobalALM} is not bounded below on the set $A$ for any
$\lambda \in \Lambda$ and $r \ge 0$.
\end{example}

If the set $A$ is compact, then the localization principle guarantees that the equality 
$\mathcal{A}(\mathcal{P}) = \mathcal{A}_{loc}(\mathcal{P})$ holds true. Thus, according to the localization principle,
it is necessary and sufficient to prove that there exists a multiplier $\lambda \in \Lambda$ such that $\lambda$ is a
local augmented Lagrange multiplier at every globally optimal solution of the problem $(\mathcal{P})$ in order to prove
the existence of a (global) augmented Lagrange multiplier. In particular, if a globally optimal solution of
$(\mathcal{P})$ is unique, and the set $A$ is compact, then a local augmented Lagrange multiplier at this optimal
solution exists if and only if a global augmented Lagrange multiplier exists.

Recall that $X$ is a topological space.

\begin{theorem}[Localization principle] \label{Thrm_LocPrinciple_Compact}
Let $A$ be compact, $\Phi$ be l.s.c. on the set $A \times \{ 0 \}$, and the function 
$\langle \lambda, \cdot \rangle$ be continuous at the origin for any $\lambda \in \Lambda$. Suppose also that $\sigma$
has a valley at zero, and the function $\mathscr{L}(\cdot, \lambda, r)$ is l.s.c. on $A$ for any 
$\lambda \in \mathcal{A}_{loc}(\mathcal{P})$ and for any $r \ge 0$ large enough. 
Then $\mathcal{A}(\mathcal{P}) = \mathcal{A}_{loc}(\mathcal{P})$.
\end{theorem}

\begin{proof}
Note that from the facts that $\Phi(x, 0) \equiv f(x)$, $\Phi$ is l.s.c. on $A \times \{ 0 \}$, and $A$ is compact it
follows that there exists a globally optimal solution of the problem $(\mathcal{P})$.

Fix an arbitrary $\lambda \in \mathcal{A}_{loc}(\mathcal{P})$. By the assumption of the theorem, there exists 
$r_0 \ge 0$ such that the function $\mathscr{L}(\cdot, \lambda, r)$ is l.s.c. on $A$ for any $r \ge r_0$. Hence and from
the fact that $A$ is compact it follows that for any $r \ge r_0$ there exists 
$x(r) \in \argmin_{x \in A} \mathscr{L}(x, \lambda, r)$. Applying the compactness of the set $A$ again, one obtains that
there exists a cluster point $x^* \in A$ of the net $\{ x(r) \}$, $r \in [r_0, + \infty)$ (see, e.g., \cite{Kelley},
Theorem~5.2). By Remark~\ref{Rmrk_MinNetAugmLagr} the point $x^*$ is a globally optimal solution of 
the problem $(\mathcal{P})$.

By the definition of $\mathcal{A}_{loc}(\mathcal{P})$ for any $\tau > r(x^*, \lambda)$ there exists a neighbourhood $U$
of $x^*$ such that
$$
  \mathscr{L}(x, \lambda, \tau) \ge \mathscr{L}(x^*, \lambda, \tau) = f(x^*) = f^* \quad \forall x \in U \cap A.
$$
Consequently, taking into account the fact the function $\mathscr{L}(x, \lambda, r)$ is non-decreasing in $r$ one gets
that $\mathscr{L}(x, \lambda, r) \ge f^*$ for all $x \in U \cap A$ and $r \ge \tau$. From the fact that $x^*$ is a
cluster point of the net $\{ x(r) \}$, $r \in [r_0, + \infty)$ one obtains that there exists $r \ge \tau$ such that
$x(r) \in U \cap A$. Therefore $\mathscr{L}(x(r), \lambda, r) \ge f^*$, which implies that $\lambda$ is an augmented
Lagrange multiplier of the problem $(\mathcal{P})$ by virtue of Proposition~\ref{Prp_ExactPenaltyRepr}, and the fact
that $x(r)$ is a globally optimal solution of $\mathscr{L}(\cdot, \lambda, r)$ on $A$.	 
\end{proof}

In order to extend Theorem~\ref{Thrm_LocPrinciple_Compact} to the case when the set $A$ is not compact, we need to
introduce the following concept of a \textit{non-degenerate} augmented Lagrangian (cf. the definition of non-degenerate
penalty function in \cite{Dolgopolik}).

\begin{definition}
Let $X$ be a normed space, and let $\lambda^* \in \Lambda$ be fixed. The augmented Lagrangian 
$\mathscr{L}(x, \lambda, r)$ is said to be \textit{non-degenerate} for $\lambda = \lambda^*$ if there exist $r_0 \ge 0$
and $K > 0$ such that for any $r \ge r_0$ the function $\mathscr{L}(\cdot, \lambda^*, r)$ attains a global minimum on
the set $A$, and there exists $x_r \in \argmin_{x \in A} \mathscr{L}(x, \lambda^*, r)$ such that $\| x_r \| \le K$.
\end{definition}

Roughly speaking, the non-degeneracy condition does not allow global minimizers of $\mathscr{L}(\cdot, \lambda, r)$ on
$A$ to escape to infinity as $r \to \infty$. Also, the non-degeneracy of the augmented Lagrangian $\mathscr{L}$ plays a
crucial role for the validity of the localization principle in the finite dimensional case.

\begin{theorem}[Localization principle] \label{Thrm_LocPrinciple_Minimizers}
Suppose that the following assumptions are satisfied:
\begin{enumerate}
\item{$X$ is a finite dimensional normed space;}

\item{$A$ is closed;}

\item{there exists a globally optimal solution of the problem $(\mathcal{P})$;}

\item{$\Phi$ is l.s.c. on the set $A \times \{ 0 \}$;}

\item{the function $\langle \lambda, \cdot \rangle$ is continuous at the origin for any $\lambda \in \Lambda$;}

\item{$\sigma$ has a valley at zero.}
\end{enumerate}
Then $\mathcal{A}(\mathcal{P})$ coincides with the set of all those $\lambda \in \mathcal{A}_{loc}(\mathcal{P})$ for
which the augmented Lagrangian $\mathscr{L}(x, \lambda, r)$ is non-degenerate.
\end{theorem}

\begin{proof}
Let $\lambda^*$ be an augmented Lagrange multiplier of $(\mathcal{P})$. Then 
$\lambda^* \in  \mathcal{A}_{loc}(\mathcal{P})$, and by Corollary~\ref{Crlr_ExactPenRepresentation} any globally optimal
solution of the problem $(\mathcal{P})$ is a global minimizer of $\mathscr{L}(\cdot, \lambda^*, r)$ on $A$ for any 
$r \ge r(\lambda^*)$. Therefore the augmented Lagrangian $\mathscr{L}(x, \lambda, r)$ is non-degenerate for 
$\lambda = \lambda^*$ with $r_0 = r(\lambda^*)$ and $K = \| x^* \|$, where $x^*$ is a globally optimal solution of the
problem $(\mathcal{P})$.

Suppose, now, that $\lambda^* \in \mathcal{A}_{loc}(\mathcal{P})$, and the augmented Lagrangian 
$\mathscr{L}(x, \lambda, r)$ is non-degenerate for $\lambda = \lambda^*$. Choose an increasing unbounded 
sequence $\{ r_n \} \subset [r_0, + \infty)$, where $r_0$ is from the definition of the non-degeneracy condition. Then
for any $n \in \mathbb{N}$ there exists $x_n \in \argmin_{x \in A} \mathscr{L}(x, \lambda^*, r_n)$ such that 
$\| x_n \| \le K$ for some $K \ge 0$. From the facts that $X$ is a finite dimensional normed space, the set $A$ is
closed, and the sequence $\{ x_n \}$ is bounded it follows that there exists a subsequence $\{ x_{n_k} \}$ converging to
a point $x^* \in A$. Moreover, by Proposition~\ref{Prp_MinSeqAugmLagr} the point $x^*$ is a globally optimal solution of
the problem $(\mathcal{P})$. 

Recall that $\lambda^* \in \mathcal{A}_{loc}(\mathcal{P})$. Therefore for any $\tau > r(x^*, \lambda^*)$ there exists a
neighbourhood $U$ of $x^*$ such that
$$
  \mathscr{L}(x, \lambda^*, \tau) \ge \mathscr{L}(x^*, \lambda^*, \tau) = f(x^*) = f^* 
  \quad \forall x \in U \cap A.
$$
Taking into account the fact that $\mathscr{L}(x, \lambda^*, r)$ is non-decreasing in $r$ one gets that
$\mathscr{L}(x, \lambda^*, r) \ge f^*$ for all $x \in U \cap A$ and $r \ge \tau$. From the facts that the subsequence 
$\{ x_{n_k} \} \subset A$ converges to $x^*$, and $\{ r_n \}$ is an increasing unbounded sequence it follows that there
exists $k \in \mathbb{N}$ such that $r_{n_k} \ge \tau$ and $x_{n_k} \in U \cap A$. Hence one has
$\mathscr{L}(x_{n_k}, \lambda^*, r_{n_k}) \ge f^*$. Consequently, applying Proposition~\ref{Prp_ExactPenaltyRepr}, and
the fact that $x_{n_k}$ is a global minimizer of $\mathscr{L}(\cdot, \lambda^*, r)$ on $A$ one obtains 
that $\lambda^* \in \mathcal{A}(\mathcal{P})$.	 
\end{proof}

Let us also give a different formulation of the localization principle in the finite dimensional case in which the
non-degeneracy condition is replaced by an assumption on a sublevel set of the augmented Lagrangian 
$\mathscr{L}(x, \lambda, r)$. 

For any $x \in X$ denote $\varphi(x) = \inf\{ \sigma(p) \mid p \in P \colon \Phi(x, p) < + \infty \}$.
The function $\varphi$ is called \textit{the penalty term} associated with the augmented Lagrangian 
$\mathscr{L}(x, \lambda, r)$. In the context of Example~\ref{Exmpl_AssociatedPenFunc}, one has
$\varphi(x) = \inf_{p \in (-G(x))} \sigma(p)$. In particular, in the case of the sharp Lagrangine one has 
$\varphi(x) = d(0, G(x))$, while in the case of the proximal Lagrangian one has $\varphi(x) = d(0, G(x))^2 / 2$.

\begin{theorem}[Localization principle] \label{Thrm_LocPrinciple_SublevelSet}
Let all assumptions of Theorem~\ref{Thrm_LocPrinciple_Minimizers} be satisfied. Suppose also that
the function $\mathscr{L}(\cdot, \lambda, r)$ is l.s.c. on $A$ for any $\lambda \in \mathcal{A}_{loc}(\mathcal{P})$ and
for any sufficiently large $r \ge 0$ (that might depend on $\lambda$). Then the set $\mathcal{A}(\mathcal{P})$ consists
of all those $\lambda \in \mathcal{A}_{loc}(\mathcal{P})$ for which there exist $r_0 \ge 0$ and $\delta > 0$ such that
the set
\begin{equation} \label{AugmLagrReducedSet}
  \Big\{ x \in A \Bigm| \mathscr{L}(x, \lambda, r_0) < f^*, \: \varphi(x) < \delta \Big\}
\end{equation}
is bounded or empty, and the function $\mathscr{L}(\cdot, \lambda, r_0)$ is bounded below on $A$.
\end{theorem}

\begin{proof}
Clearly, if $\lambda \in \mathcal{A}(\mathcal{P})$, then $\lambda \in \mathcal{A}_{loc}(\mathcal{P})$, and, by
Proposition~\ref{Prp_ExactPenaltyRepr}, for any $r \ge r(\lambda)$ and $\delta > 0$ the set \eqref{AugmLagrReducedSet}
is empty, and $\mathscr{L}(x, \lambda, r) \ge f^*$ for all $x \in A$.

Suppose, now, that $\lambda \in \mathcal{A}_{loc}(\mathcal{P})$, and there exist $r_0 \ge 0$ and $\delta > 0$ such that
the set \eqref{AugmLagrReducedSet} is bounded or empty, and the function $\mathscr{L}(\cdot, \lambda, r_0)$ is bounded
below on $A$. Note that for any $x \in A$ and $r \ge r_0$ one has
$$
  \Phi(x, p) - \langle \lambda, p \rangle + r \sigma(p) \ge
  \Phi(x, p) - \langle \lambda, p \rangle + r_0 \sigma(p) + (r - r_0) \varphi(x) \quad \forall p \in P,
$$
which implies that $\mathscr{L}(x, \lambda, r) \ge \mathscr{L}(x, \lambda, r_0) + (r - r_0) \varphi(x)$ for 
all $x \in A$. Therefore for any $x \in A$ such that $\varphi(x) \ge \delta$ one has
$$
  \mathscr{L}(x, \lambda, r ) \ge c + (r - r_0) \delta \ge f^* \quad \forall r \ge \tau := r_0 + \frac{f^* - c}{\delta},
$$
where $c = \inf_{x \in A} \mathscr{L}(x, \lambda, r_0) > - \infty$. Hence for any $r \ge \tau$ one has
$$
  \Big\{ x \in A \Bigm| \mathscr{L}(x, \lambda, r) < f^* \Big\} \subset
  \Big\{ x \in A \Bigm| \mathscr{L}(x, \lambda, r_0) < f^*, \: \varphi(x) < \delta \Big\}.
$$
Therefore there exists $K \ge 0$ such that for any $r \ge \tau$ one has that $\| x \| \le K$ for all 
$x \in \{ y \in A \mid \mathscr{L}(y, \lambda, r) < f^* \}$. Consequently, applying the lower semicontinuity of the
function $\mathscr{L}(\cdot, \lambda, r)$ on the set $A$, and the fact that $X$ is a finite dimensional normed space,
one obtains that the function $\mathscr{L}(\cdot, \lambda, r)$ attains a global minimum on the set $A$ for any
sufficiently large $r \ge 0$. Furthermore, for any $r$ large enough either $\mathscr{L}(x, \lambda, r) \ge f^*$ for all
$x \in A$, which implies that $\lambda \in \mathcal{A}(\mathcal{P})$ due to Proposition~\ref{Prp_ExactPenaltyRepr}, or
any global minimizer $x^*$ of $\mathscr{L}(\cdot, \lambda, r)$ on $A$ satisfies the inequality $\| x^* \| \le K$. In the
latter case, one obtains that $\mathscr{L}(x, \lambda, r)$ is non-degenerate. Then applying
Theorem~\ref{Thrm_LocPrinciple_Minimizers} one gets the desired result.
\end{proof}

\begin{remark} \label{Rmrk_GeneralPrinciple}
Theorems~\ref{Thrm_LocPrinciple_Compact}--\ref{Thrm_LocPrinciple_SublevelSet} allow one to understand a general
principle behind many similar results on augmented Lagrangian functions (see~\cite{SunLiMckinnon},
Theorems~3.1--3.4; \cite{LiuYang}, Theorem~3.3; \cite{ZhouXiuWang}, Theorem~3.3; \cite{ZhouZhouYang2014}, Theorems~3.1
and 3.2; \cite{ZhouChen2015}, Theorem~3.1, etc.). Namely, a multiplier $\lambda^*$ is a global augmented Lagrange
multiplier if and only if $\lambda^*$ is a local augmented Lagrange multiplier at every global minimizer of the problem
$(\mathcal{P})$ and a certain compactness assumption holds true (note that this principle can also be formulated in
terms of global and local saddle points). We call this general result \textit{the localization principle}, since it
allows one to perform a local analysis of the augmented Lagrangian in order to obtain global existence results.
\end{remark}

Let us demonstrate that, without some additional assumptions (such as the compactness of the set $A$), the localization
principle is valid only in the finite dimensional case.

\begin{example}
Let $X = A = \ell_2$, $P = \Lambda = \mathbb{R}$ and $\langle \lambda, p \rangle \equiv \lambda p$. Define
$\omega_1(t) = \max\{ - t + 1, -1 \}$, and for any $n \ge 2$ set
$$
  \omega_n(t) = \begin{cases}
    - 3 / n, & \text{if } t \in (- \infty, - 1 - 2 / n), \\
    t + 1 + 1 / n - n (t + 1)^2 / 2, & \text{if } t \in [- 1 - 2 / n, - 1], \\
    t + 1 + 1 / n, & \text{if } t \in (-1, 0], \\
    1 + 1 / n, & \text{if } t \in (0, + \infty).
  \end{cases}
$$
Finally, for any $x = (x_1, x_2, \ldots) \in \ell_2$ define $f_0(x) = \inf_{n \in \mathbb{N}} \omega_n(x_n)$, and
$h(x) = \| x \| - 1$, where $\| \cdot \|$ is the standard norm in $\ell_2$. Note that the function $f_0$ (as well as the
function $h$) is globally Lipschitz continuous and bounded below on $\ell_2$ due to the facts that for any 
$n \in \mathbb{N}$ the function $\omega_n$ is globally Lipschitz continuous with a Lipschitz constant $L \le 3$, and 
$\omega_n(t) \ge - 3/2$ for all $t \in \mathbb{R}$.

Consider the following optimization problem
$$
  \min f_0(x) \quad \text{subject to} \quad h(x) = 0.
$$
One can easily verify that $x^* = (1, 0, 0, \ldots)$ is a unique globally optimal solution of this problem, and 
$f^* = f_0(x^*) = 0$. Let $\Phi(x, p)$ be the standard dualizing parameterization for the problem above, and let
$\sigma(p) = p^2 / 2$. Then $\mathscr{L}(x, \lambda, r) = f_0(x) + \lambda h(x) + r h(x)^2 / 2$. It is easy to see
that for any $\lambda \in \mathbb{R}$ and $r \ge 0$ the function $\mathscr{L}(\cdot, \lambda, r)$ is coercive, and
Lipschitz continuous on any bounded subset of $\ell_2$.

Let us check that $\lambda^* = 1$ is a local augmented Lagrange multiplier at $x^*$. Indeed, from the facts that 
$\omega_1(t) \le 0.5$ for any $t \in [0.5, 1.5]$ and $\omega_n(t) > 0.5$ for any $t \in [-0.5, 0.5]$ it follows that
$f_0(x) = \omega_1(x_1)$ for all $x \in B(x^*, 0.5)$, where $B(x^*, 0.5)$ is the ball with centre $x^*$ and radius
$0.5$. Hence for any $x \in B(x^*, 0.5)$ and $r \ge 0$ one has
\begin{multline} \label{FirstFuncMin}
  \mathscr{L}(x, 1, r) = \omega_1(x_1) + h(x) + \frac{r}{2} h(x)^2 =
  \max\{ -x_1 + 1, -1 \} + (\| x \| - 1) \\
  + \frac{r}{2} \big( \| x \| - 1 \big)^2 \ge
  \max\big\{ - x_1 + \| x \|, \| x \| - 2 \big\} \ge 0 = f_0(x^*).
\end{multline}
Consequently, $\lambda^* = 1$ is a local augmented Lagrange multiplier at $x^*$ and $r(x^*, \lambda^*) = 0$.
Furthermore, one can verify that $\lambda^* = 1$ is a \textit{unique} augmented Lagrange multiplier at $x^*$.

Let us find a global minimum of the function $\mathscr{L}(\cdot, 1, r)$ for any $r \ge 1$. Taking into account the
fact that $f_0(x) = \inf_{n \in \mathbb{N}} \omega_n(x_n)$ one obtains that it is sufficient to find a global minimizer
$x^{(n)}$ of the function $\theta_n(x) = \omega_n(x_n) + h(x) + r h(x)^2 / 2$, $n \ge 2$ (note that from
\eqref{FirstFuncMin} it follows that $x^{(1)} = x^*$), and, then, to minimize $\mathcal{L}(x^{(n)}, 1, r)$ with respect
to $n$. 

Fix arbitrary $n \ge 2$, $x \in \ell_2$ and $r \ge 1$. If $x_n \ge 0$, then
$$
  \theta_n(x) = 1 + \frac{1}{n} + \| x \| - 1 + \frac{r}{2}(\| x \| - 1)^2 \ge \frac{1}{n}.
$$
If $x_n \in [-1, 0]$, then
$$
  \theta_n(x) = x_n + 1 + \frac{1}{n} + \| x \| - 1 + \frac{r}{2}(\| x \| - 1)^2 \ge
  x_n + \| x \| + \frac{1}{n} \ge \frac{1}{n}.
$$
If $x_n \le - 1 - 2 / n$, then
\begin{multline*}
  \theta_n(x) = - \frac{3}{n} + \| x \| - 1 + \frac{r}{2}(\| x \| - 1)^2 \\
  \ge - \frac{3}{n} + \inf_{t \ge 1 + 2 / n} \Big( t - 1 + \frac{r}{2} (t - 1)^2 \Big) = 
  - \frac{1}{n} + \frac{2r}{n^2}.
\end{multline*}
Finally, if $x_n \in [-1 - 2/n, -1]$, then
\begin{multline*}
  \theta_n(x) = x_n + 1 + \frac{1}{n} - \frac{n}{2} (x_n + 1)^2 + \| x \| - 1 + \frac{r}{2}(\| x \| - 1)^2 \ge
  - \frac{n}{2} (x_n + 1)^2 \\
  + \frac{r}{2}(\| x \| - 1)^2 + \frac{1}{n} \ge
  - \frac{n}{2} (x_n + 1)^2 + \frac{r}{2}(|x_n| - 1)^2 + \frac{1}{n} \ge 
  \min\left\{ \frac{1}{n}, -\frac{1}{n} + \frac{2r}{n^2} \right\}.
\end{multline*}
Thus, for any $n \ge 2$ and $r \ge 1$ one has that $\inf_{x \in \ell_2} \theta_n(x) = 1 / n$, if $n \le r$, 
$\inf_{x \in \ell_2} \theta_n(x) = (2r - n) / n^2$, if $n > r$, and the infimum is attained at the point $x^{(n)}$ such
that $x^{(n)}_k = 0$ for all $k \ne n$, $x^{(n)}_n = -1$, if $n \le r$, and $x^{(n)}_n = - 1 - 2 / n$, if $n > r$. Hence
for any $r \ge 1$ one has
$$
  \inf_{x \in \ell_2} \mathscr{L}(x, 1, r) = \inf_{n \ge 2} \mathscr{L}(x^{(n)}, 1, r) =
  \inf_{n \ge 2 r} \left( -\frac{1}{n} + \frac{2r}{n^2} \right) < 0.
$$
Furthermore, from the fact that the last infimum above is attained for some $n_0 \in \mathbb{N}$ such that either
$n_0 \le 4r \le n_0 + 1$ or $n_0 - 1 \le 4r \le n_0$ it follows that the function $\mathscr{L}(\cdot, 1, r)$ attains a
global minimum at the point $x(r) = x^{(n_0)}$, and $\mathscr{L}(x(r), 1, r) < 0$, while $f^* = f_0(x^*) = 0$. Therefore
$\lambda^* = 1$ is a not an augmented Lagrange multiplier of $(\mathcal{P})$.

Note that $\| x(r) \| \le 2$ for all $r \ge 1$. Therefore the augmented Lagrangian $\mathscr{L}(x, \lambda, r)$ is
non-degenerate for $\lambda = 1$. Furthermore, as it was mentioned above, the function $\mathscr{L}(\cdot, \lambda, r)$
is coercive and Lipschitz continuous on any bounded subset of $\ell_2$. Hence, as it is easy to see, all assumptions of
Theorems~\ref{Thrm_LocPrinciple_Minimizers} and \ref{Thrm_LocPrinciple_SublevelSet}, apart from the assumptions that
$X$ is finite dimensional, are satisfied. Thus, without some additional assumptions, the localization principle does not
hold true in the infinite dimensional case.

Let us also note that $|h(x)| = \inf\{ \| x - y \| \mid y \in \ell_2 \colon h(y) = 0 \}$ (i.e. $h(x)$ has a
\textit{global} error bound; furthermore, it is easy to see that $h(x)$ is \textit{globally} metrically regular, i.e.
metrically regular on $\ell_2 \times h(\ell_2)$), and the function $f_0$ is \textit{globally} Lipschitz continuous.
Consequently, by \cite{Dolgopolik}, Proposition~3.16, the penalty function $F(x, r) = f_0(x) + r |h(x)|$ is exact.
Therefore by Proposition~\ref{Prp_SharpLagrEquivExactPenalty}, the multiplier $\lambda^* = 1$ is an augmented Lagrange
multiplier of the sharp Lagrangian. However, as it was shown above, $\lambda^* = 1$ is not an augmented Lagrange
multiplier of the proximal Lagrangian. It is worth mentioning that this kind of result is impossible in the finite
dimensional case, since if some $\lambda^* \in \Lambda$ is a local augmented Lagrange multiplier of both sharp and
proximal Lagrangians at all globally optimal solutions, and both sharp and proximal Lagrangians are coercive, then by
Theorem~\ref{Thrm_LocPrinciple_SublevelSet} the multiplier $\lambda^*$ is a global augmented Lagrange multiplier of both
sharp and proximal Lagrangians. 
\end{example}

\section{Applications of the localization principle}
\label{Sect_Applications}

In this section, we consider applications of the general theory developed above to the mathematical programming and
nonlinear semidefinite programming problems. Our aim is to show that the existence of a local augmented Lagrange
multiplier can be proved via sufficient optimality conditions, while the existence of a (global) augmented Lagrange
multiplier can be easily proved with the use of the localization principle.

\subsection{Mathematical programming}

Let $X = \mathbb{R}^d$. Consider the mathematical programming problem of the form
\begin{equation} \label{NonlinearProgram}
  \min f_0(x) \quad \text{s.t.} \quad g_i(x) = 0, \quad i \in I, \quad g_j(x) \le 0, \quad j \in J, 
  \quad x \in A.
\end{equation}
Here $f_0, g_s \colon \mathbb{R}^d \to \mathbb{R}$, $s \in I \cup J$ are given functions, $I = \{ 1, \ldots, m_1 \}$,
$J = \{ m_1 + 1, \ldots, m_2 \}$, and $A \subset \mathbb{R}^d$ is a nonempty closed set. For the sake of simplicity, we
suppose that the set $A$ is convex. 

Let $\Lambda = P = \mathbb{R}^{m_2}$, $\langle \cdot, \cdot \rangle$ be the inner product in $\mathbb{R}^s$, 
$s \in \mathbb{N}$, $\| \cdot \|$ be the Euclidean norm, and $\sigma(p) = \| p \|^2 / 2$. Define
$\Phi(x, p) = f_0(x)$, if $g_i(x) + p_i = 0$ for all $i \in I$, and $g_j(x) + p_j \le 0$ for all $j \in J$, and
$\Phi(x, p) = + \infty$, otherwise. Let also $f(x) = \Phi(x, 0)$. Then the problems \eqref{NonlinearProgram} and
$(\mathcal{P})$ coincide.

As it is easy to verify, for any $x \in \mathbb{R}^d$, $\lambda \in \mathbb{R}^{m_2}$ and $r > 0$ one has
\begin{multline*}
  \mathscr{L}(x, \lambda, r) = f_0(x) + 
  \sum_{i = 1}^{m_1} \left( \lambda_i g_i(x) + \frac{r}{2} g_i(x)^2 \right) \\
  + \sum_{j = m_1 + 1}^{m_2} \Bigg( \lambda_j \max\left\{ g_j(x), - \frac{\lambda_j}{r} \right\}
  + \frac{r}{2} \max\left\{ g_j(x), - \frac{\lambda_j}{r} \right\}^2 \Bigg),
\end{multline*}
i.e. $\mathscr{L}$ is the proximal Lagrangian (the Hestenes-Powell-Rockafellar Lagrangian) for the problem
\eqref{NonlinearProgram}. 

At first, we study the local behaviour of the augmented Lagrangian $\mathscr{L}$. Denote by 
$L(x, \lambda) = f_0(x) + \sum_{s = 1}^{m_2} \lambda_s g_s(x)$ the standard Lagrangian for the problem
\eqref{NonlinearProgram}, and denote by $\Omega$ the feasible set of this problem. Let the functions $f_0$ and $g_s$ be
twice differentiable at a point $x^* \in \Omega$. Recall that the pair $(x^*, \lambda^*)$ with $\lambda^* \in
\mathbb{R}^{m_2}$ is called a \textit{KKT pair} of the problem \eqref{NonlinearProgram} if $\lambda_j^* g_j(x^*) = 0$
and $\lambda_j^* \ge 0$ for all $j \in J$, and $\big\langle D_x L(x^*, \lambda^*), v \big\rangle \ge 0$ for all 
$v \in T_A(x^*)$, where $T_A(x^*)$ is the contingent cone to the set $A$ at $x^*$. Note that if $(x^*, \lambda^*)$ is a
KKT pair, then $\mathscr{L}(x^*, \lambda^*, r) = f(x^*)$ for all $r \ge 0$. We say that that a KKT pair 
$(x^*, \lambda^*)$ satisfies the second order sufficient optimality condition if the matrix
$D^2_{xx} L(x^*, \lambda^*)$ is positive definite on the cone 
\begin{multline*}
  K(x^*, \lambda^*) = \Big\{ v \in T_A(x^*) \Bigm| \langle \nabla g_s(x^*), v \rangle = 0, \: s \in I \cup
  J_+(x^*, \lambda^*), \\
  \langle \nabla g_j(x^*), v \rangle \le 0, \: j \in J_0(x^*, \lambda^*), \quad
  \langle D_x L(x^*, \lambda^*), v \rangle = 0 \Big\},
\end{multline*}
where $J_+(x^*, \lambda^*) = \{ j \in J_0(x^*) \mid \lambda_j^* > 0 \}$,
$J_0(x^*, \lambda^*) = \{ j \in J_0(x^*) \mid \lambda_j^* = 0 \}$, and
$J_0(x^*) = \{ j \in J \mid g_j(x^*) = 0 \}$.
The following theorem slightly improves all analogous results on the Hestenes-Powell-Rockafellar Lagrangian (cf.
\cite{Bazaraa}, Theorem~9.3.3; \cite{SunLiMckinnon}, Theorem~2.1; \cite{LiuTangYang}, Theorem~2; \cite{WangZhouXu},
Theorem~2.3; \cite{LuoMastroeniWu}, Theorems~3.1 and 3.2; \cite{ZhouXiuWang}, Theorem~2.8; 
\cite{ZhouZhouYang2014}, Proposition~3.1), since we consider all types of constraints (i.e. equality, inequality and
nonfunctional constraints), do not use any constraint qualifications, and utilize weaker sufficient optimality
conditions than in the aforementioned papers.

\begin{theorem} \label{Thrm_NonlinProgr_LocalAugmMult}
Let $x^*$ be a locally optimal solution of the problem \eqref{NonlinearProgram}. Suppose that the functions $f_0$ and
$g_s$, $s \in I \cup J$, are twice differentiable at the point $x^*$, and a KKT pair $(x^*, \lambda^*)$ satisfies the
second order sufficient optimality condition. Then $\lambda^*$ is a local augmented Lagrange multiplier of
$(\mathcal{P})$ at $x^*$.
\end{theorem}

\begin{proof}
With the use of the second order Taylor expansion of the functions $f_0$ and $g_s$ at $x_*$ one gets that for any 
$r > 0$ there exists a neighbourhood $U_r$ of $x_*$ such that
\begin{multline} \label{AugmLagr_TaylorExpans_Gen}
  \bigg| \mathscr{L}(x, \lambda^*, r) - \mathscr{L}(x^*, \lambda^*, r)
  - \big\langle D_x L(x^*, \lambda^*), x - x^* \big\rangle \\
  - \frac{1}{2} \big\langle x - x^*, D^2_{xx} L(x^*, \lambda^*) (x - x^*) \big\rangle
  - \frac{r}{2} \sum_{s \in I \cup J_+(x^*, \lambda^*)} \big\langle \nabla g_s(x^*), x - x^* \big\rangle^2 \\
  - \frac{r}{2} \sum_{j \in J_0(x^*, \lambda^*)} \max\big\{ \langle \nabla g_j(x^*), x - x^* \rangle, 0 \big\}^2
  \bigg| \le \frac{1}{r} \| x - x^* \|^2
\end{multline}
for any $x \in U_r$. Here we utilised the fact that the twice differentiability of the functions $g_j$, $j \in J$, at
$x^*$, and the complementary slackness condition imply that for any $x$ in a sufficiently small neighbourhood of $x^*$
one has
$$
  \max\left\{ g_j(x), - \frac{\lambda_j^*}{r} \right\} = \begin{cases}
    g_j(x), & \text{if } j \in J_+(x^*, \lambda^*), \\
    \max\{ g_j(x), 0 \}, & \text{if } j \in J_0(x^*, \lambda^*), \\
    0, & \text{if } j \in J \setminus J_0(x^*),
  \end{cases}
$$
and for all $j \in J_0(x^*, \lambda^*)$ one has
$$
  \max\{ g_j(x), 0 \}^2 = \max\big\{ \langle \nabla g_j(x^*), x - x^* \rangle, 0 \big\}^2 + o( \| x - x^* \|^2 ),
$$
where $o( \| x - x^* \|^2 ) / \| x - x^* \|^2 \to 0$ as $x \to x^*$.

Arguing by reductio ad absurdum, suppose that $\lambda^*$ is not a local augmented Lagrange multiplier at $x^*$. Then
for any $n \in \mathbb{N}$ there exists $x_n \in A \cap U_n$ such that 
$\mathscr{L}(x_n, \lambda^*, n) < \mathscr{L}(x^*, \lambda^*, n)$. Hence applying \eqref{AugmLagr_TaylorExpans_Gen} one
obtains that for any $n \in \mathbb{N}$ the following inequality holds true
\begin{multline} \label{AugmLagr_TaylorExpans}
  0 > \big\langle D_x L(x^*, \lambda^*), \Delta x_n \big\rangle +
  \frac12 \big\langle \Delta x_n, D^2_{xx} L(x^*, \lambda^*) \Delta x_n \big\rangle \\
  + \frac{n}{2} \sum_{s \in I \cup J_+(x^*, \lambda^*)} \big\langle \nabla g_s(x^*), \Delta x_n \big\rangle^2 +
  \frac{n}{2} \sum_{j \in J_0(x^*, \lambda^*)} \max\big\{ \langle \nabla g_j(x^*), \Delta x_n \rangle, 0 \big\}^2 \\
  - \frac{1}{n} \| \Delta x_n \|^2,
\end{multline}
where $\Delta x_n = x_n - x^*$. Denote $v_n = \Delta x_n / \| \Delta x_n \|$. Clearly, without loss of generality one
can suppose that the sequence $\{ v_n \}$ converges to a vector $v^*$ such that $\| v^* \| = 1$. Furthermore, 
$v^* \in T_A(x^*)$ by virtue of the fact that $x_n \in A$ for all $n \in \mathbb{N}$.

From \eqref{AugmLagr_TaylorExpans} it follows that
$$
  0 > \big\langle D_x L(x^*, \lambda^*), \Delta x_n \big\rangle +
  \frac12 \big\langle \Delta x_n, D^2_{xx} L(x^*, \lambda^*) \Delta x_n \big\rangle - \frac{1}{n} \| \Delta x_n \|^2
$$
for any $n \in \mathbb{N}$. Dividing this inequality by $\| \Delta x_n \|$ and passing to the limit
as $n \to \infty$ one obtains that $0 \ge \langle D_x L(x^*, \lambda^*), v^* \rangle$. Consequently, applying
the fact that $(x^*, \lambda^*)$ is a KKT pair one gets that $\langle D_x L(x^*, \lambda^*), v^* \rangle = 0$.

From the facts that $x_n \in A$ by definition, and $A$ is convex it follows that $x_n - x^* \in T_A(x^*)$ for all 
$n \in \mathbb{N}$. Hence taking into account the fact that $(x^*, \lambda^*)$ is a KKT pair one gets that 
$\langle D_x L(x^*, \lambda^*), \Delta x_n \rangle \ge 0$ for all $n \in \mathbb{N}$. Applying this lower estimate in
\eqref{AugmLagr_TaylorExpans}, dividing by $\| \Delta x_n \|^2$, and passing to the limit as $n \to \infty$ one obtains
that $\langle \nabla g_s(x^*), v^* \rangle = 0$ for all $s \in I \cup J_+(x^*, \lambda^*)$, and
$\langle \nabla g_j(x^*), v^* \rangle \le 0$ for all $j \in J_0(x^*, \lambda^*)$. Thus, $v^* \in K(x^*, \lambda^*)$.

Applying \eqref{AugmLagr_TaylorExpans}, and the inequality $\langle D_x L(x^*, \lambda^*), \Delta x_n \rangle \ge 0$
one obtains that
$$
  0 > \frac12 \big\langle \Delta x_n, D^2_{xx} L(x^*, \lambda^*) \Delta x_n \big\rangle -
  \frac{1}{n} \| \Delta x_n \|^2 \quad \forall n \in \mathbb{N}.
$$
Dividing this inequality by $\| \Delta x_n \|^2$, and passing to the limit as $n \to \infty$, one gets that 
$0 \ge \langle v^*, D^2_{xx} L(x^*, \lambda^*) v^* \rangle$, which contradicts the fact that the pair $(x^*, \lambda^*)$
satisfies the second order sufficient optimality condition.
\end{proof}

Now, we can obtain simple necessary and sufficient conditions for the existence of an augmented Lagrange multiplier of
\eqref{NonlinearProgram}. Denote by $\Omega^* \subset \Omega$ the set of globally optimal solutions of this problem. We
suppose that $\Omega^* \ne \emptyset$.

It is easy to see that if $\lambda^*$ is an augmented Lagrange multiplier of the problem \eqref{NonlinearProgram}, and
the functions $f$ and $g_s$ are twice differentiable, then for any $x^* \in \Omega^*$ the pair $(x^*, \lambda^*)$ is a
KKT pair that satisfies the second order necessary optimality condition:
$\langle v, D^2_{xx} L(x^*, \lambda^*) v \rangle \ge 0$ for all $v \in K(x^*, \lambda^*)$.
Thus, for the existence of an augmented Lagrange multiplier of \eqref{NonlinearProgram} it is necessary that there
exists a multiplier $\lambda^* \in \mathbb{R}^{m_2}$ such that for any $x^* \in \Omega^*$ the pair $(x^*, \lambda^*)$ is
a KKT pair that satisfies the second order \textit{necessary} optimality condition. It should be noted that, in the
general case, there might not exist a multiplier $\lambda^*$ that satisfies the KKT conditions at every globally optimal
solution of the problem \eqref{NonlinearProgram}. In particular, if the problem \eqref{NonlinearProgram} has at least
two global minimizers with disjoint sets of Lagrange multipliers, then an augmented Lagrange multiplier of this problem
does not exist.

\begin{remark}
Since the existence of an augmented Lagrange multiplier implies the existence of a KKT-point of the problem
\eqref{NonlinearProgram}, the problem of existence of augmented Lagrange multipliers is closely related to the
study of KKT optimality conditions for nonlinear programming problems. In particular, any necessary condition for the
existence of a KKT-point is also a necessary conditions for the existence of an augmented Lagrange multiplier.
For some recent developments on KKT conditions see, e.g.,
\cite{YangMeng2007,MengYang2010,FloresBazanMastroeni} and references therein.
\end{remark}

Under some additional assumptions, one simply has to replace ``necessary optimality condition'' by
``sufficient optimality condition'' in order to obtain sufficient conditions for the existence of an augmented Lagrange
multiplier.

\begin{theorem} \label{Thrm_NonlinProgr_GlobalAugmMult}
Let $A$ be closed, $f_0$ be l.s.c. on $A$, and $g_s$, $s \in I \cup J$, be continuous on $A$. Suppose also that the
functions $f_0$ and $g_s$, $s \in I \cup J$, are twice differentiable at every point $x^* \in \Omega^*$, and there
exists a multiplier $\lambda^* \in \mathbb{R}^{m_2}$ such that for any $x^* \in \Omega^*$ the pair $(x^*, \lambda^*)$ is
a KKT pair satisfying the second order sufficient optimality condition. Then the following statements are equivalent:
\begin{enumerate}
\item{$\lambda^*$ is an augmented Lagrange multiplier of the problem \eqref{NonlinearProgram};}

\item{there exist $r_0 > 0$ and $K > 0$ such that for any $r \ge r_0$ there exists 
$x_r \in \argmin_{x \in A} \mathscr{L}(x, \lambda^*, r)$ with $\| x_r \| \le K$;
}

\item{there exist $r_0 > 0$ and $\delta > 0$ such that the function $\mathscr{L}(\cdot, \lambda^*, r_0)$ is bounded
below on $A$, and the set
\begin{equation} \label{SublevelSet_NonlinearProgram}
  \Big\{ x \in A \Bigm| \mathscr{L}(x, \lambda^*, r_0) < f^*, \: 
  \varphi(x) < \delta \Big\}
\end{equation}
is bounded or empty, where $\varphi(x) = \sum_{i \in I} g_i(x)^2 + \sum_{j \in J} \max\{ 0, g_j(x) \}^2$.
\label{Assump_SubLevelSetBoundedness}}
\end{enumerate}
In particular, an augmented Lagrange multiplier of \eqref{NonlinearProgram} exists, provided one of the following
additional assumptions is satisfied:
\begin{enumerate}
\item{$A$ is compact;
}

\item{$f_0$ is coercive on the set $A$, i.e. $f_0(x_n) \to + \infty$ as $n \to \infty$ for any sequence 
$\{ x_n \} \subset A$ such that $\| x_n \| \to +\infty$ as $n \to \infty$;
}

\item{the function $f_0(\cdot) + r_0 \varphi(x)$ is coercive on $A$ for some $r_0 \ge 0$;
}

\item{the function $\varphi(x)$ is coercive on $A$, and the function $\mathscr{L}(\cdot, \lambda^*, r_0)$ is bounded
below on $A$ for some $r_0 > 0$.
}
\end{enumerate}
\end{theorem}

\begin{proof}
By Theorem~\ref{Thrm_NonlinProgr_LocalAugmMult}, the multiplier $\lambda^*$ is a local augmented Lagrange multiplier at
every $x^* \in \Omega^*$. Then applying the localization principle (Theorems~\ref{Thrm_LocPrinciple_Minimizers} and
\ref{Thrm_LocPrinciple_SublevelSet}) one obtains the required result.	 
\end{proof}

\begin{remark}
{(i)~Note that the set \eqref{SublevelSet_NonlinearProgram} is bounded for some $r_0 > 0$ and $\delta > 0$, in
particular, if for some $x_0 \in \Omega$ the set
$\{ x \in A \mid \mathscr{L}(x, \lambda^*, r_0) < f_0(x_0), \: \varphi(x) < \delta \}$ is bounded. Furthermore,
one can easily verify that the set \eqref{SublevelSet_NonlinearProgram} is bounded for some 
$r_0 > 0$ and $\delta > 0$, if the set $\{ x \in A \mid f_0(x) < f^* + \alpha, \: \varphi(x) < \delta \}$ is
bounded for some $\alpha > 0$ and $\delta > 0$.
}

\noindent{(ii)~The theorem above provides first simple \textit{necessary and sufficient} conditions for the existence of
an augmented Lagrange multiplier of the Hestenes-Powell-Rockafellar Lagrangian function for a mathematical programming
problem that does not rely on any assumptions on the optimal value functions or optimal/suboptimal solutions of a
perturbed problem (cf.~\cite{ShapiroSun2004}). Furthermore, Theorem~\ref{Thrm_NonlinProgr_GlobalAugmMult} improves all
similar results on the existence of augmented Lagrange multipliers 
\cite{SunLiMckinnon,LiuTangYang,WangZhouXu,LuoMastroeniWu,ZhouXiuWang,ZhouZhouYang2014}, since it is based on weaker
sufficient optimality conditions, is formulated as \textit{necessary and sufficient} conditions, and does not require
any additional assumptions (such as the compactness of the set $A$ as in \cite{SunLiMckinnon}, the uniqueness of a
globally optimal solution as in \cite{LuoMastroeniWu}, the boundedness of the least exact penalty parameters at globally
optimal solutions as in \cite{LiuTangYang,WangZhouXu} or the compactness of certain sublevel sets as in
\cite{LiuTangYang,WangZhouXu,LuoMastroeniWu,ZhouXiuWang,ZhouZhouYang2014}).
}
\end{remark}

Theorems~\ref{Thrm_NonlinProgr_LocalAugmMult} and \ref{Thrm_NonlinProgr_GlobalAugmMult} can be easily extended to 
the case when the functions $f_0$ and $g_s$ are not twice differentiable, but are $C^{1, 1}$ function. In particular,
let us extend Theorem~\ref{Thrm_NonlinProgr_LocalAugmMult} to this more general case.

For the sake of simplicity, let $A = \mathbb{R}^d$. Fix $x^* \in \Omega$, and suppose that the functions
$f_0$ and $g_s$ are $C^{1, 1}$ at $x^*$, i.e. suppose that $f_0$ and $g_s$ are differentiable near
$x^*$, and their gradients are Lipschitz continuous in a neighbourhood of this point. Recall that if a function $h$ is
$C^{1, 1}$ at $x^*$, then \textit{the generalized Hessian matrix} $\partial^2 h(x^*)$ of $h$ at $x^*$ is defined as the
convex hull of the set of all those matrices $M$ for which there exists a sequence $\{ x_n \}$ such that
$x_n \to x^*$ and $\nabla^2 h(x_n) \to M$ as $n \to \infty$ (see, e.g., \cite{HiriartUrruty,KlatteTammer}). One
can verify that $\partial^2 h(x^*)$ is a nonempty compact convex set of symmetric matrices.

Let $(x^*, \lambda^*)$ be a KTT pair satisfying the strict complementary slackness condition, i.e. 
$J_0(x^*, \lambda^*) = \emptyset$. We say that the pair $(x^*, \lambda^*)$ satisfies the \textit{generalized} second
order sufficient optimality condition if each $M \in \partial^2_{xx} L(x^*, \lambda^*)$ is positive definite on 
the subspace 
$\{ v \in \mathbb{R}^d \mid \langle \nabla g_s(x^*), v \rangle = 0, \: s \in I \cup J_+(x^*, \lambda^*)\}$
(see~\cite{KlatteTammer}). Here $\partial^2_{xx} L(x^*, \lambda^*)$ is the generalized Hessian matrix of the function
$L(\cdot, \lambda^*)$ at $x^*$. Similarly to the case when the functions $f_0$ and $g_s$ are twice differentiable, the
validity of the generalized second order sufficient optimality guarantees that $\lambda^*$ is a local augmented Lagrange
multiplier.

\begin{theorem} \label{Thrm_NonlinProgr_LocalAugmMult_C11}
Let $x^*$ be a locally optimal solution of the problem \eqref{NonlinearProgram}, and the functions $f_0$ and $g_s$, 
$s \in I \cup J$, be $C^{1,1}$ at the point $x^*$. Suppose that a KKT pair $(x^*, \lambda^*)$ satisfies the strict
complementary slackness condition, and the generalized second order sufficient optimality condition. Then $\lambda^*$ is
a local augmented Lagrange multiplier of $(\mathcal{P})$ at $x^*$.
\end{theorem}

\begin{proof}
From the facts that the functions $g_s$ are continuous at $x^*$, and the strict complementary slackness condition holds
true it follows that there exists a neighbourhood $U$ of $x^*$ such that for any $x \in U$ and $r > 0$ one has
$$
  \mathscr{L}(x, \lambda^*, r) = L(x^*, \lambda^*) + 
  \frac{r}{2} \sum_{s \in I \cup J_+(x^*, \lambda^*)} g_s(x)^2.
$$
Consequently, applying the sum rule for generalized Hessian matrices (\cite{HiriartUrruty}, Theorem~2.2) one obtains
that the set $\partial^2_{xx} \mathscr{L}(x^*, \lambda^*, r)$ is contained in the set
$H(x^*, \lambda^*, r) = \partial^2_{xx} L(x^*, \lambda^*) + 
r \sum_{s \in I \cup J_+(x^*, \lambda^*)} \nabla g_s(x^*) \nabla g_s(x^*)^T$.
Taking into account the fact that the generalized second order sufficient optimality condition holds true, one can
easily verify that there exists $r_0 > 0$ such that for any $r \ge r_0$ all matrices $M \in H(x^*, \lambda^*, r)$ are
positive definite (see~\cite{LiuYang}, Lemma~3.1). Hence applying the second-order sufficient optimality condition for
$C^{1, 1}$ functions (\cite{KlatteTammer}, Theorem~1) one obtains that $x^*$ is a point of local minimum of 
the function $\mathscr{L}(\cdot, \lambda^*, r)$ for any $r \ge r_0$. Therefore $\lambda^*$ is a local augmented
Lagrange multiplier at $x^*$.	 
\end{proof}

\subsection{Nonlinear semidefinite programming}

Let us note that Theorems~\ref{Thrm_NonlinProgr_LocalAugmMult}--\ref{Thrm_NonlinProgr_LocalAugmMult_C11} can be easily
extended to the case of second-order cone programming, cone constrained optimization, nonlinear semidefinite
programming, semi-infinite programming and other constrained optimization problems. One simply has to prove the
existence of a local augmented Lagrange multiplier (or a local saddle point) with the use of sufficient optimality
conditions as in \cite{ZhouChen2015,ZhouZhouYang2014,WuLuoYang,WangZhouXu}, and then apply the localization principle in
order to obtain simple necessary and sufficient conditions for the existence of an augmented Lagrange multiplier that
strengthen all similar results existing in the literature. In order to illustrate this statement, let us consider the
following nonlinear semidefinite programming problem
\begin{equation} \label{NonlinearSemiDefProg}
  \min f_0(x) \quad \text{subject to} \quad G(x) \preceq 0, \quad h(x) = 0,
\end{equation}
where $G \colon \mathbb{R}^d \to \mathbb{S}^m$ and $h = (h_1, \ldots, h_l) \colon \mathbb{R}^d \to \mathbb{R}^l$ are
given function, $\mathbb{S}^m$ denotes the set of all $m \times m$ real symmetric matrices, and the relation 
$G(x) \preceq 0$ means that the matrix $G(x)$ is negative semidefinite. Denote by $Tr(\cdot)$ the trace of a matrix, by 
$A \bullet B = Tr(AB)$ the inner product of $A, B \in \mathbb{S}^m$, and by $\| A \|_F = \sqrt{Tr(A^2)}$ the Frobenius
norm of $A \in \mathbb{S}^m$.

Let $\Lambda = P = \mathbb{S}^m \times \mathbb{R}^l$, and 
$\langle \lambda, p \rangle = \mu \bullet q + \nu^T w$ for any 
$\lambda = (\mu, \nu) \in \Lambda$ and $p = (q, w) \in P$. Define $\Phi(x, p) = f_0(x)$, if
$G(x) + q \preceq 0$ and $h(x) + w = 0$; $\Phi(x, p) = + \infty$, otherwise, and
$\sigma(p) = 0.5 \| q \|^2_F + 0.5 \| w \|^2$. Then one can verify (see~(2.9) in \cite{ShapiroSun2004}) that for any 
$x \in \mathbb{R}^d$, $\lambda \in \Lambda$ and $r > 0$ one has
$$
  \mathscr{L}(x, \lambda, r) = f_0(x) + \frac{1}{2r} \Big( Tr\big( [r G(x) + \mu]_+^2 \big) - Tr(\mu^2) \Big) + 
  \nu^T h(x) + \frac{r}{2} \| h(x) \|^2,
$$
where $[\cdot]_+$ denotes the projection of a matrix onto the cone of $m \times m$ positive semidefinite matrices.

In order to obtain necessary and sufficient conditions for the existence of an augmented Lagrange multiplier for the
problem \eqref{NonlinearSemiDefProg}, let us recall KKT optimality conditions for this problem
\cite{Shapiro_SemiDef,WuLuoYang}. Let $x^*$ be a locally optimal solution of the problem \eqref{NonlinearSemiDefProg},
and the functions $f_0$, $G$ and $h$ be twice differentiable at $x^*$. A pair $(x^*, \lambda^*)$, where $\lambda^* =
(\mu^*, \nu^*) \in \Lambda$, is called \textit{a KKT pair} of the problem \eqref{NonlinearSemiDefProg}, if $\mu^*$ is
positive semidefinite, $\mu^* G(x^*) = 0$ and $D_x L(x^*, \lambda^*) = 0$, where 
$L(x, \lambda) = f_0(x) + \mu \bullet G(x) + \nu^T h(x)$ is the classical Lagrangian. Suppose that $\rank(G(x^*)) < m$.
One says that a KKT pair $(x^*, \lambda^*)$ satisfies \textit{the second order sufficient optimality condition}, if the
matrix
$$
  D^2_{xx} L(x^*, \lambda^*) - 
  2 \Big[ \mu^* \bullet \Big( D_{x_i} G(x^*) G(x^*)^{\dagger} D_{x_j} G(x^*)  \Big) \Big]_{i, j = 1}^d
$$
is positive definite on the cone
$$
  C(x^*) = \Big\{ v \in \mathbb{R}^d \Bigm| 
  \sum_{i = 1}^d v_i E_0^T D_{x_i} G(x^*) E_0 \preceq 0, \: \nabla h(x^*) v = 0, \:
  \nabla f_0(x^*)^T v = 0 \Big\},
$$
where $G(x^*)^{\dagger}$ is the Moore-Penrose pseudoinverse of the matrix $G(x^*)$, and $E_0$ is 
$m \times (m - \rank(G(x^*)))$ matrix composed from the eigenvectors of $G(x^*)$ corresponding to its zero eigenvalue.
Note that if $\rank(G(x^*)) = m$, then the constraint $G(x) \preceq 0$ is inactive at $x^*$, i.e. $x^*$ is a locally
optimal solution of the problem of minimizing $f_0(x)$ subject to $h(x) = 0$. In this case, $\mu^* = 0$, and we utilize
sufficient optimality conditions for the problem \eqref{NonlinearProgram}.

Denote by $\Omega^*$ the set of optimal solutions of the problem \eqref{NonlinearSemiDefProg}. We suppose 
that $\Omega^* \ne \emptyset$. The following result strengthens Theorem~4 in \cite{WuLuoYang}, since we do not assume
that an optimal solution of \eqref{NonlinearSemiDefProg} is unique, and, more importantly, we obtain \textit{necessary
and sufficient} conditions for the existence of an augmented Lagrange multiplier, in contrast to only
\textit{sufficient} conditions in \cite{WuLuoYang}.

\begin{theorem}
Let $f_0$ be l.s.c., and $G$ and $h$ be continuous. Let also the functions $f_0$, $G$ and $h$ be twice
differentiable at every point $x^* \in \Omega^*$, and let there exists a multiplier $\lambda^* \in \Lambda$ such that
for any $x^* \in \Omega^*$ the pair $(x^*, \lambda^*)$ is a KKT pair satisfying the second order sufficient optimality
condition. Then $\lambda^*$ is an augmented Lagrange multiplier of the problem \eqref{NonlinearSemiDefProg} if and only
if there exist $r_0 > 0$ and $\delta > 0$ such that the function $\mathscr{L}(\cdot, \lambda^*, r_0)$ is bounded below,
and the set
\begin{equation} \label{SublevelSet_SemiDefProg}
  \Big\{ x \in \mathbb{R}^d \Bigm| \mathscr{L}(x, \lambda^*, r_0) < f^*, \: \varphi(x) < \delta \Big\}
\end{equation}
is bounded or empty, where $\varphi(x) = Tr([G(x)]_+^2) + \| h(x) \|^2$.
\end{theorem}

\begin{proof}
By Theorem~\ref{Thrm_NonlinProgr_LocalAugmMult}, and \cite{WuLuoYang}, Theorem~3, the multiplier $\lambda^*$ is a local
augmented Lagrange multiplier at every $x^* \in \Omega^*$. Applying the localization principle
(Theorem~\ref{Thrm_LocPrinciple_SublevelSet}) one obtains the desired result.	 
\end{proof}

\begin{remark}
One can verify that the set \eqref{SublevelSet_SemiDefProg} is bounded, if there exist $\alpha > 0$ and $\delta > 0$
such that the set $\{ x \in \mathbb{R}^d \mid f_0(x) < f^* + \alpha, \: \varphi(x) < \delta \}$ is bounded. 
\end{remark}

\bibliographystyle{abbrv}  
\bibliography{AugmLagrMult}

\end{document}